\documentclass[a4paper,11pt]{amsart}

\usepackage{amsmath}
\usepackage{amssymb}
\usepackage{amsthm}
\usepackage{tikz}
\usepackage{verbatim}
\usepackage[justification=centering]{caption}

\newtheorem{theorem}{Theorem}[section]
\newtheorem{proposition}[theorem]{Proposition}
\newtheorem{lemma}[theorem]{Lemma}
\newtheorem{corollary}[theorem]{Corollary}
\newtheorem*{theo}{Theorem}

\theoremstyle{definition}
\newtheorem{definition}[theorem]{Definition}

\newtheorem{remark}[theorem]{Remark}

\numberwithin{equation}{section}

\newcommand{\Q}{\mathbb{Q}}

\newcommand{\Z}{\mathbb{Z}}
\newcommand{\Zt}{\mathbb{Z}[t^{\pm1}]}
\newcommand{\K}{\mathbb{K}}
\newcommand{\F}{\mathbb{F}}
  
\newcommand{\Dd}{\mathcal{D}}

\newcommand{\Bb}{\mathcal{B}}
\newcommand{\Cc}{\mathcal{C}}
\newcommand{\Cb}{\mathcal{C}_{\scriptstyle{\partial }}}
\newcommand{\Ccl}{\overline{\mathcal{C}}}

\renewcommand{\L}{\mathcal{L}}
\newcommand{\J}{\mathcal{J}}
\newcommand{\abX}{\widehat{X}}

\newcommand{\Int}{\mathrm{Int}}

\definecolor{vert}{RGB}{0,205,0}
\definecolor{bviolet}{rgb}{0.54,0.17,0.89}
\definecolor{aqua}{rgb}{0.0,1.0,1.0}

\begin{document}

\title[The algebraic topology of $4$--manifolds multisections]{The algebraic topology of $4$--manifolds multisections}

\author{Delphine Moussard}
\author{Trenton Schirmer}

\begin{abstract}
 A multisection of a 4--manifold is a decomposition into 1--handlebodies intersecting pairwise along 3--dimensional handlebodies or along a central closed surface; this generalizes the Gay--Kirby trisections. We show how to compute the twisted absolute and relative homology, the torsion and the twisted intersection form of a 4--manifold from a multisection diagram. The homology and torsion are given by a complex of free modules defined by the diagram and the intersection form is expressed in terms of the intersection form on the central surface. We give efficient proofs, with very few computations, thanks to a retraction of the (possibly punctured) 4--manifold onto a CW--complex determined by the multisection diagram. Further, a multisection induces an open book decomposition on the boundary of the 4--manifold; we describe the action of the monodromy on the homology of the page from the multisection diagram.
 \smallskip
 
\end{abstract}

\maketitle

\section{Introduction and main results}

A trisection is a type of combinatorial structure on 4--manifolds which was discovered by Gay and Kirby via Morse 2--functions \cite{GayKirby}. They proved that any smooth $4$--manifold, possibly with boundary, can be decomposed as the union of three $4$--dimensional $1$--handlebodies, with $3$--dimensional $1$--handlebodies as pairwise intersections and a compact surface as global intersection. Such a trisection can be described by a diagram, namely the central surface with collections of curves that define the $3$--dimensional pieces. A trisection diagram determines a smooth $4$--manifold up to diffeomorphism, so that one should be able to read topological invariants of the manifold on the diagram. In the setting of closed $4$--manifolds, Feller, Klug, Schirmer and Zemke \cite{FKSZ} provided a computation of the homology and intersection form of the manifold from a trisection diagram, and Florens and Moussard \cite{FM} derived the twisted homology and torsion, and the twisted intersection form. Following these papers, Tanimoto \cite{Tanimoto} computed the homology of $4$--manifolds with connected boundary. Here we recover and generalize these results, computing from a diagram the twisted absolute and relative homology and torsion and the twisted intersection form for any trisected $4$--manifold with boundary. Moreover, we work with ``multisections'' in the sense of Islambouli--Naylor \cite{IN}, namely a cyclic decomposition of the manifold into any number of $4$--dimensional $1$--handlebodies, where successive pieces meet along $3$--dimensional $1$--handlebodies while non-successive ones meet along the central surface. We propose a more efficient approach. While Feller--Klug--Schirmer--Zemke worked with a handle decomposition of the manifold underlying the trisection, Florens--Moussard directly used the datum of the trisection. This last method reduced the homological computations, but the computation of torsion was quite intricate. Here we consider a deformation-retraction of the (possibly punctured) manifold onto a CW--complex associated with the multisection diagram. This simplifies the computations and provides the torsion ``for free''. This retraction could be useful for further computations of homological or homotopical invariants.

A multisection of a $4$--manifold $X$ with boundary induces an open book decomposition on the boundary. The monodromy of this open book has been described algorithmically by Castro, Gay and Pinz\'on-Caicedo \cite{CGPC1} from a diagram. Here we derive the action of the monodromy on the homology of the page from which can be derived a computation of the homology of $\partial X$ as well as the Alexander module of the binding determined by the monodromy.
\medskip

For $4$--manifolds with boundary, the handlebodies of a multisection inherit (hyper) compression bodies structures related to the way they intersect the boundary of the manifold.

\begin{definition}\label{def:CompressionBody}
 A \emph{compression body} $C$ is a cobordism from a compact orientable surface $\partial_-C$ to a connected compact orientable surface $\partial_+C$ which is constructed using only $1$--handles. Likewise a \emph{hyper compression body} $V$ is a cobordism from a compact orientable $3$--manifold $\partial_-V$ to a connected compact orientable $3$--manifold $\partial_+V$ constructed using only $1$--handles. A \emph{lensed} (hyper) compression body is then obtained by collapsing the vertical boundary of the cobordism so that the boundary of $\partial_+C$ ($\partial_+V$) becomes identified with the boundary of $\partial_-C$ ($\partial_-V$). 
\end{definition}

In the case that $\partial_-C=\emptyset$, it is understood at $C$ is built using only $1$--handles attached to a single $0$--handle. A (lensed) compression body is \emph{trivial} if $\partial_-C\cong \partial_+C$.  This means it is just a thickened surface $S\times I$, or if lensed, it is obtained from $S\times I$ by collapsing the $I$--fibers of $\partial S\times I$. 

\begin{definition}\label{def:Multisection}
 A \emph{multisection} of a compact orientable $4$--manifold $X$ is a decomposition $X=X_1\cup\cdots \cup X_n$ into $4$--dimensional $1$--handlebodies $X_i$ with the following properties (all arithmetic involving indices is mod $n$):
 \begin{enumerate}
  \item each $X_i$ has a lensed hyper compression body structure such that $\partial_-X_i=X_i\cap \partial X$ and if $\partial X\neq \emptyset$, there is a fixed surface $\Sigma_\partial$ such that, for all $1\leq i\leq n$, $\partial_-X_i$ is diffeomorphic to the trivial lensed compression body obtained by pinching the vertical boundary of $\Sigma_\partial\times I$,
  \item $\displaystyle\Sigma=\bigcap_{i=1}^n X_i$ is a compact connected orientable surface,
  \item $C_i=X_i\cap X_{i+1}$ is a $3$--dimensional $1$--handlebody with a lensed compression body structure satisfying $\partial_+C_i=\Sigma$ and $\partial_-C_i=C_i\cap \partial X\cong \Sigma_\partial$ for all $i$,
  \item when $|i-j|>1$, $X_i\cap X_j=\Sigma$.
 \end{enumerate}
 A multisection is called a \emph{trisection} when $n=3$.
\end{definition}

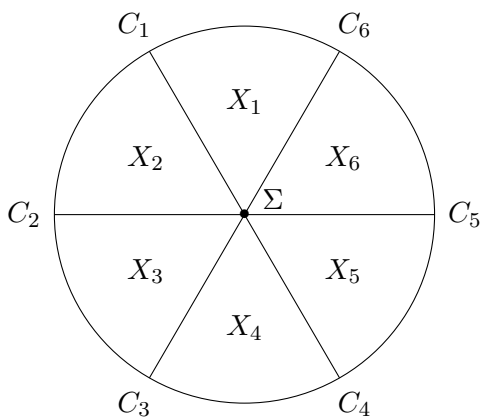
\begin{figure}[htb]
\begin{center}
\begin{tikzpicture} [scale=0.5]
 \draw (0,0) circle (5);
 \foreach \s in{1,...,5,6} {
 \draw[rotate=60*(\s+1)] (0,0) -- (5,0);
 \draw[rotate=60*(\s+1)] (5.8,0) node {$C_\s$};
 \draw[rotate=60*\s+30] (3,0) node {$X_\s$};}
 \draw (0,0) node {$\scriptstyle\bullet$} (0.2,-0.12) node[above right] {$\Sigma$};
\end{tikzpicture}
\end{center}
\caption{Schematic of a multisection}
\label{fig:multisection}
\end{figure}

The condition that the $C_i$ are $1$--handlebodies implies that $\Sigma$ is closed if and only if $X$ is closed, and $\Sigma_\partial$ has no closed component. 
This is the framework of most of the literature on trisections and within this framework a unified calculation of the algebraic topology is possible. The specific case where more general compression bodies are allowed, {\em ie} the case in which $\Sigma_\partial$ has closed components, was considered in the original paper of Gay and Kirby \cite{GayKirby}; however, the calculations become more delicate and require special treatment.  We postpone the homology computations in this case to a forthcoming publication in order to avoid the extra complications here. 

In the case that $\partial X\neq \emptyset$, it is also to be understood that for all $i$ mod $n$, $\partial_-X_i$ is parameterized as $\Sigma_\partial\times I/\sim$ in such a way that $\partial_-C_{i-1}=\Sigma_\partial\times \{0\}$ and $\partial_-C_i=\Sigma_\partial\times\{1\}$. Therefore, the multisection induces an open book decomposition on $\partial X$ with page $\Sigma_\partial$.

We fix once and for all a multisected manifold $X=\cup_{1\leq i\leq n}X_i$ and set $C_i=X_i\cap X_{i+1}$, $\Sigma=\cap_iX_i$. 

\begin{definition}
 Let $C$ be a compression body. A \emph{defining collection of disks} for $C$ is a collection $\Dd$ of disks properly embedded in $C$ such that $C\setminus\eta(\Dd)$ is a thickening of $\partial_-C$ (for instance the co-core disks of the 1--handles in the definition). The boundary $\partial\Dd\subset\partial_+C$ is a {\em defining collection of curves} for $C$.
\end{definition}

\begin{definition}
 A \emph{diagram} of the multisection $X=\cup_{1\leq i\leq n}X_i$ is a tuple $(\Sigma;c_1,\ldots , c_n)$ where~$c_i$ is a defining collection of curves for~$C_i$. 
\end{definition} 

A multisection diagram determines a unique smooth $4$--manifold \cite{CGPC2}. 
The structure of the $X_i$ gives some constraints on the curves of a multisection diagram. For each $i$, $X_i$ is obtained from a thickened $\partial_-X_i$ by attaching $1$-handles, so that $\partial_+X_i\cong (S^2\times S^1)^{\sharp k}\#(\#\partial_-X_i)$, where $k$ is the number of $1$-handles in excess of the minimum required to connect $\partial_-X_i$ and $\# \partial_-X_i$ is the connected sum of all components of $\partial_- X_i$. Now Definition~\ref{def:Multisection} implies that $C_{i-1}\cup_\Sigma C_i$ is a sutured Heegaard splitting of $\partial_+X_i$, so that the Heegaard diagram $(\Sigma;c_{i-1},c_i)$ is always handleslide-diffeomorphic to a standard diagram as represented in Figure~\ref{fig:StandardDiagram}.

\begin{figure}[htb]
\begin{center}
\begin{tikzpicture} [scale=0.3]
\newcommand{\trou}{
(2,0) ..controls +(0.5,-0.25) and +(-0.5,-0.25) .. (4,0)
(2.3,-0.1) ..controls +(0.6,0.2) and +(-0.6,0.2) .. (3.7,-0.1)}
\draw (0,0) ..controls +(0,1) and +(-2,1) .. (4,2);
\draw (4,2) ..controls +(1,-0.5) and +(-1,0) .. (6,1.25);
\draw[dashed] (6,1.25);
\draw (0,0) ..controls +(0,-1) and +(-2,-1) .. (4,-2);
\draw (4,-2) ..controls +(1,0.5) and +(-1,0) .. (6,-1.25);
\foreach \x/\y in {6/0,18/0,38/3} {
\begin{scope} [xshift=\x cm,yshift=\y cm]
\draw (0,1.25) ..controls +(1,0) and +(-2,1) .. (4,2);
\draw (4,2) ..controls +(2,-1) and +(-2,-1) .. (8,2);
\draw (8,2) ..controls +(2,1) and +(-1.2,0) .. (12,1.25);
\draw (0,-1.25) ..controls +(1,0) and +(-2,-1) .. (4,-2);
\draw (4,-2) ..controls +(2,1) and +(-2,1) .. (8,-2);
\draw (8,-2) ..controls +(2,-1) and +(-1.2,0) .. (12,-1.25);
\end{scope}}
\foreach \x in {0,6,12,18,24} {
\draw[xshift=\x cm] \trou;}
\foreach \x in {0,6} {
\draw[color=red,xshift=\x cm] (3,-0.2) ..controls +(0.2,-0.5) and +(0.2,0.5) .. (3,-2.3);
\draw[dashed,color=red,xshift=\x cm] (3,-0.2) ..controls +(-0.2,-0.5) and +(-0.2,0.5) .. (3,-2.3);
\draw[color=blue,xshift=\x cm] (3,0)ellipse(1.6 and 0.8);}
\foreach \x in {11.9,17.9,23.9} {
\draw[color=red,xshift=\x cm] (3,-0.2) ..controls +(0.2,-0.5) and +(0.2,0.5) .. (3,-2.3);
\draw[dashed,color=red,xshift=\x cm] (3,-0.2) ..controls +(-0.2,-0.5) and +(-0.2,0.5) .. (3,-2.3);}
\foreach \x in {12.2,18.2,24.2} {
\draw[color=blue,xshift=\x cm] (3,-0.2) ..controls +(0.2,-0.5) and +(0.2,0.5) .. (3,-2.3);
\draw[dashed,color=blue,xshift=\x cm] (3,-0.2) ..controls +(-0.2,-0.5) and +(-0.2,0.5) .. (3,-2.3);}
\foreach \x/\y in {38/3,44/3,38/-3,38/9} {
\draw[xshift=\x cm,yshift=\y cm] \trou;}
\foreach \x/\y in {44/-3,50/3}{
\begin{scope} [xshift=\x cm,yshift=\y cm]
\draw (0,1.25) ..controls +(0.4,-1) and +(0.4,1) .. (0,-1.25);
\draw[dashed] (0,1.25) ..controls +(-0.4,-1) and +(-0.4,1) .. (0,-1.25);
\end{scope}}
\begin{scope} [xshift=26cm,yshift=-3cm]
\draw (14,2) ..controls +(2,1) and +(-1.2,0) .. (18,1.25);
\draw (14,-2) ..controls +(2,-1) and +(-1.2,0) .. (18,-1.25);
\draw (12,1.25) ..controls +(0.5,0) and +(-1,-0.5) .. (14,2);
\draw (12,-1.25) ..controls +(0.5,0) and +(-1,0.5) .. (14,-2);
\end{scope}
\begin{scope} [xshift=26cm,yshift=9cm]
\draw (14,2) ..controls +(2,1) and +(-1,0) .. (18,2);
\draw (14,-2) ..controls +(2,-1) and +(-1,0) .. (18,-2);
\draw (12,1.25) ..controls +(0.5,0) and +(-1,-0.5) .. (14,2);
\draw (12,-1.25) ..controls +(0.5,0) and +(-1,0.5) .. (14,-2);
\end{scope}
\begin{scope} [xshift=26cm,yshift=-9cm]
\draw (12,1.25) ..controls +(2,0) and +(-2,0) .. (18,2);
\draw (12,-1.25) ..controls +(2,0) and +(-2,0) .. (18,-2);
\end{scope}
\foreach \y in {9,-9} {
\begin{scope} [xshift=44cm,yshift=\y cm]
\foreach \s in {1,-1} {
\draw (0,2*\s) ..controls +(0.2,-0.5*\s) and +(0.2,0.5*\s) .. (0,0.5*\s);
\draw[dashed] (0,2*\s) ..controls +(-0.2,-0.5*\s) and +(-0.2,0.5*\s) .. (0,0.5*\s);}
\draw (0,0.5) arc (90:270:0.5);
\end{scope}}
\foreach \y in {-6,0,6} {
\draw (38,1.75+\y) arc (90:270:1.75);}
\foreach \s in {1,-1} {
\draw (30,1.25*\s) .. controls +(3,0) and +(-5,0) .. (38,10.25*\s);}
\foreach \y in {-9,-3,3} {
\foreach \x/\c in {37.9/red,38.2/blue} {
\begin{scope} [xshift=\x cm,yshift=\y cm,\c]
\draw (0,1.25) ..controls +(0.4,-1) and +(0.4,1) .. (0,-1.25);
\draw[dashed] (0,1.25) ..controls +(-0.4,-1) and +(-0.4,1) .. (0,-1.25);
\end{scope}}}
\end{tikzpicture}
\end{center} \caption{Heegaard diagram for $C_{i-1}\cup C_i$\\[2pt]{\footnotesize In this example, $C_{i-1}$ and $C_i$ are constructed with eight $1$--handles and $X_i$ with six $1$--handles.\\ The manifold $X$ has four boundary components. The components of the page $\Sigma_\partial$ have a pair \mbox{(genus,number of boundary components)} equal to $(1,2)$, $(2,1)$, $(1,1)$ and $(0,2)$.}} \label{fig:StandardDiagram}
\end{figure}
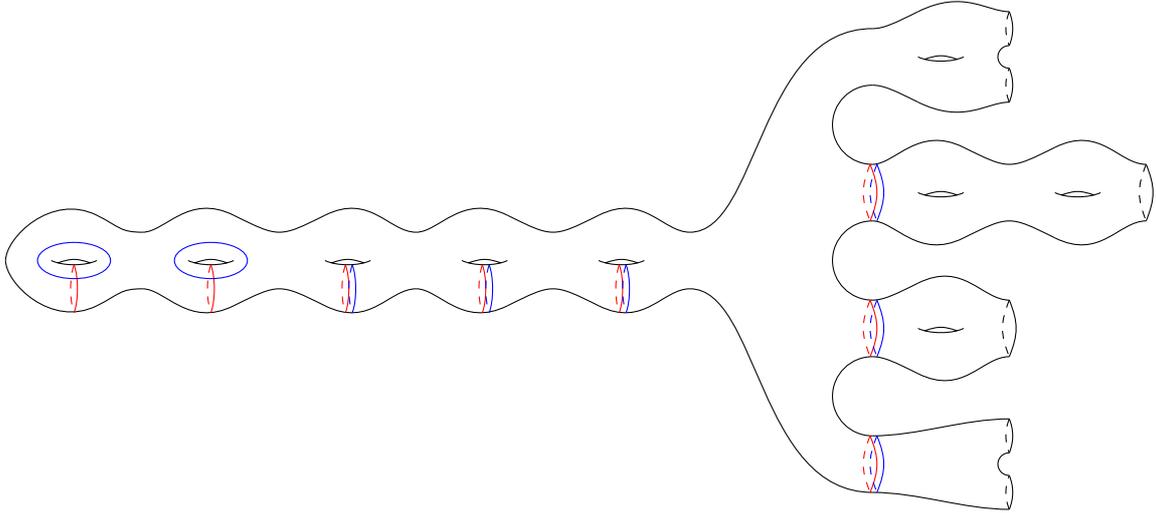

Fix a homomorphism $\varphi:\Z[\pi_1(X)]\rightarrow R$ where $R$ is a commutative ring. We shall express the absolute and relative homology of $X$, twisted by $\varphi$, in terms of the multisection diagram. Fix a point $*\in\Int(\Sigma)$ and let $L_i^\varphi$ be the submodule of $H_1^\varphi(\Sigma,*)$ generated by the homology classes of the curves in $c_i$. 

\begin{theo}[Theorem~\ref{thm:XHomology}]
 If $\partial X\neq\emptyset$, the twisted homology of $X$ is given by the chain complex of free $R$--modules
 \begin{equation} \tag{$\Cc$}
  0\rightarrow \bigoplus_{i=1}^n(L_{i-1}^\varphi\cap L_i^\varphi)\xrightarrow{\partial_2} \bigoplus_{i=1}^n L_i^\varphi\xrightarrow{\partial_1} H_1^\varphi(\Sigma,*)\xrightarrow{\partial_0} H_0^\varphi(*)
 \end{equation}
 where $\partial_2\big((x_i)_{1\leq i\leq n}\big)=\big((x_i-x_{i+1})_{1\leq i\leq n}\big)$ and $\partial_1\big((x_i)_{1\leq i\leq n}\big)=\sum_{i=1}^n x_i$.
 Moreover, if $R$ is a field, there are complex bases $c$ of \ref{complexforX} such that $\tau^\varphi(X;h)=\tau(\Cc;c,h)$ for any homology basis $h$ of~$X$ and \ref{complexforX}. 
\end{theo}

Let $\Sigma'$ be the surface $\Sigma$ with a small open disk removed, such that the point $*$ lies on the boundary of the removed disk. For $1\leq i\leq n$, let $\J_i^\varphi$ be the orthogonal complement in $H_1^\varphi(\Sigma',\partial\Sigma)$ of $L_i^\varphi$ with respect to the twisted intersection form on $H_1^\varphi(\Sigma,*)\times H_1^\varphi(\Sigma',\partial\Sigma)$.

\begin{theo}[Theorem~\ref{thm:RelXHomology}]
 If $\partial X\neq\emptyset$, the twisted homology of $(X,\partial X)$ is given by the chain complex of free $R$--modules
 \begin{equation} \tag{$\Cb$}
  H_2^\varphi(\Sigma,\Sigma')\xrightarrow{\partial_3} \bigoplus_i(\J_{i-1}^\varphi\cap \J_{i}^\varphi)\xrightarrow{\partial_2} \bigoplus_i\J_i^\varphi\xrightarrow{\partial_1} H_1^\varphi(\Sigma',\partial\Sigma)\rightarrow 0
 \end{equation}
 where $\partial_3([\Sigma])=[\partial(\Sigma\setminus\Sigma')]$, $\partial_2((x_i)_{1\leq i\leq n})=((x_i-x_{i+1})_{1\leq i\leq n})$ and $\partial_1((x_i)_{1\leq i\leq n})=\sum_{i=1}^nx_i$.
 Moreover, if $R$ is a field, there is a complex basis $c$ of \ref{complexXrel} such that $\tau^\varphi(X,\partial X;h)=\tau(\Cb;c,h)$ for any homology basis~$h$ of~$(X,\partial X)$ and \ref{complexXrel}. 
\end{theo}

When $\Sigma$ is closed, we define the $L_i^\varphi$ in $H_1^\varphi(\Sigma',*)$. These are lagrangians, namely they are their own orthogonal complement with respect to the intersection form.

\begin{theo}[Theorem~\ref{thm:XClosedHomology}]
 If $X$ is closed, the twisted homology of $X$ is given by the chain complex of free $R$--modules
 \begin{equation} \tag{$\Ccl$}
  H_2^\varphi(\Sigma,\Sigma')\xrightarrow{\partial_3} \bigoplus_i(L_{i-1}^\varphi\cap L_{i}^\varphi)\xrightarrow{\partial_2} \bigoplus_i L_i^\varphi\xrightarrow{\partial_1} H_1^\varphi(\Sigma',*)\rightarrow H_0^\varphi(*)
 \end{equation}
 where $\partial_3([\Sigma])=[\partial\Sigma']$, $\partial_2((x_i)_{1\leq i\leq n})=((x_i-x_{i+1})_{1\leq i\leq n})$ and $\partial_1((x_i)_{1\leq i\leq n})=\sum_{i=1}^nx_i$.
 Moreover, if $R$ is a field, there is a complex basis $c$ of \ref{complexXclosed} such that $\tau^\varphi(X;h)=\tau(\Ccl;c,h)$ for any homology basis~$h$ of~$X$ and \ref{complexXclosed}. 
\end{theo}

These three theorems allow to represent homology classes by mainly explicit chains in the multisected manifold, which meet transversely along copies of the central surface $\Sigma$. This provides a simple description of the intersection form on $X$.

\begin{theo}[Theorem~\ref{thm:IntersectionForm}]
 Suppose $h_1=[(x_i)_{1\leq i\leq n}]\in H_2^\varphi(X)$ with $(x_i)_{1\leq i\leq n}\in\oplus_i L_i^\varphi$ and $h_2=[(y_i)_{1\leq i\leq n}]\in H_2^\varphi(X,\partial X)$ with $(y_i)_{1\leq i\leq n}\in\oplus_i\J_i^\varphi$. Then 
 \[\langle h_1,h_2\rangle_X^\varphi=\sum_{1\leq i<j\leq n}\langle x_i,y_j\rangle_\Sigma^\varphi\] 
 where $\langle\cdot,\cdot\rangle_{X}^\varphi$ and $\langle\cdot,\cdot\rangle_\Sigma^\varphi$ are the equivariant intersection forms on $H_2^\varphi(X)\times H_2^\varphi(X,\partial X)$ and $H_1^\varphi(\Sigma,*)\times H_1^\varphi(\Sigma',\partial \Sigma)$ (on $H_2^\varphi(X)$ and $H_1(\Sigma',*)$ if $X$ is closed).
\end{theo}

\begin{theo}[Theorem~\ref{thIntFormOdd}]
 Suppose that either $h_1\in H_1^\varphi(X)$ corresponds to the element $a\in H_1^\varphi(\Sigma,*)$ and $h_2\in H_3^\varphi(X,\partial X)$ corresponds to the element $b\in \cap_i\J_i^\varphi$, or $h_1\in H_1^\varphi(X,\partial X)$ corresponds to the element $a\in H_1^\varphi(\Sigma',\partial \Sigma)$ and $h_2\in H_3^\varphi(X)$ corresponds to the element $b\in \cap_i L_i^\varphi$ ($a\in H_1^\varphi(\Sigma',*)$ if $X$ is closed). Then \[\langle h_1,h_2\rangle_X^\varphi=\langle a,b\rangle_\Sigma^\varphi.\]
\end{theo}

\subsection*{Plan of the paper}
In Section~\ref{secprelim}, we recall the definitions of twisted homology, torsion and twisted intersection form. Sections~\ref{sec:AbsHomology} and~\ref{sec:RelHomology} are devoted to the twisted homology and torsion of a $4$--manifold with non-empty boundary, respectively absolute and relative. In Section~\ref{sec:IntForm}, we describe the intersection forms. Section~\ref{sec:HomologyClosed} treats the case of a closed $4$--manifold. Section~\ref{sec:Boundary} deals with the boundary: action in homology of the monodromy of the open book and homology of the boundary. Finally, in Section~\ref{sec:Exs}, we treat some examples.

\subsection*{Convention}
The notation we set above is assumed to be fixed for the remainder of the paper.  That is, $X$ is always multisected by $n$ hyper compression bodies $X_i$ which meet in compression bodies $C_i$, all of which are attached radially about the central fiber $\Sigma$.  Additionally, $Y=C_1\cup \cdots \cup C_n$ shall be referred to as the \emph{spine} of the multisection. Also, $\varphi:\Z[\pi_1(X)]\rightarrow R$ is a homomorphism to a commutative ring $R$. 
Throughout the paper, if $Z$ is a subset of a manifold~$M$, $\eta(Z)$ denotes a regular neighborhood of $Z$ in $M$.

\subsection*{Acknowledgments}
This work has its source in the Spring Trisectors Workshop organized in 2019 at Georgia University by David Gay, Jeffrey Meier and Alexander Zupan; we are deeply grateful to the organizers. We also thank the participants, especially David Gay and Peter Feller, for motivating conversations.

\section{Algebraic Preliminaries} \label{secprelim}

\subsection{Algebraic torsion}

We recall the algebraic setup, see \cite{Mi66} and \cite{Tu01} for further details and references. Let $\K$ be a field. If $V$ is a finite dimensional $\K$--vector space and $b$ and $c$ are two bases of $V$, we denote by $[b/c]$ 
the determinant of the matrix expressing the basis change from $b$ to $c$. The bases $b$ and $c$ are {\em equivalent} if $[b/c]=1$. Let $\Cc$ be a finite complex of finite dimensional $\K$--vector spaces:
$$
\Cc=( \Cc_m\ \xrightarrow{\partial_{m}}\ \Cc_{m-1}\ \longrightarrow  \ \cdots\ \xrightarrow{\partial_1}\ \Cc_0 ).
$$
A \emph{complex basis} of $\Cc$ is a family  $c=(c_m,\dots,c_0)$ where $c_i$ is a basis of  $\Cc_i$ for all $i\in\{0,\dots,m\}$. A \emph{homology basis} of $\Cc$ is a family $h=(h_m,\dots,h_0)$ where $h_i$ is a basis of the homology group $H_i(\Cc)$ for all $i\in\{0,\dots,m\}$.
If we have chosen a basis $b_j$ of the space of $j$--dimensional boundaries $B_j(\Cc)= \operatorname{Im} \partial_{j+1}$ for all $j\in \{0,\dots,m-1\}$, then a homology basis $h$ of $\Cc$ induces an equivalence class of bases $(b_ih_i)b_{i-1}$ of $\Cc_i$ for all $i$. 

The \emph{torsion} of the $\K$--complex $\Cc$, equipped with a complex basis $c$ and a homology basis $h$, is the scalar:
\begin{displaymath} 
\tau(\Cc;c,h)=  \prod_{i=0}^m \big[(b_ih_i)b_{i-1}/c_i\big]^{(-1)^{i+1}} \in \K^*.
\end{displaymath}
It is easily checked that this definition does not depend on the choice of $b_0,\dots,b_m$. 

\subsection{Twisted homology and Reidemeister torsion} 

Let $\widehat{X}$ denote the universal cover of $X$, and for any $Z\subset X$, let $\widehat{Z}$ denote the inverse image of $Z$ under the covering map $\widehat{X}\rightarrow X$ ($\widehat{Z}$ will usually not be the universal cover of $Z$). Then $\pi_1(X)$ acts on both $\abX$ and $\widehat{Z}$ by deck transformations, and we let $\varphi: \Z[\pi_1(X)]\rightarrow R$ denote a morphism into a commutative ring $R$, so that the chain groups are $$C_i^\varphi(X,Z)=C_i(\abX,\widehat{Z})\otimes_{\Z[\pi_1(X)]}R.$$  The corresponding homology groups are denoted $H_*^\varphi(X,Z)$. Note that $H_*^\varphi(X,Z)$ can be interpreted as the homology of the pair $(\overline X,\overline Z)$, where $\overline X$ is the cover of $X$ associated to $\ker\varphi$ and $\overline Z$ is the inverse image of $Z$ under the covering map $\overline X\rightarrow X$.

Now assume $R$ is a field. 
Let $\widehat{c}$ be a complex basis of the complex of free $\Z[\pi_1(X)]$--module $C(\widehat{X},\widehat Z)$ obtained by lifting each relative cell of $(X,Z)$ to $\widehat{X}$. Then $c=\widehat{c}\otimes 1$ is a complex basis of $C^{\varphi}(X,Z)$. 
\begin{definition}
Given a homology basis $h$ of $H^\varphi(X,Z)$, the torsion of $(X,Z;\varphi)$ is 
$$ \tau^\varphi(X,Z;h) = \tau(C^{\varphi}(X,Z);c,h) \in R /\pm \varphi(\pi_1(X)).$$
\end{definition}

The ambiguity in $\pm \varphi(\pi_1(X))$ is due to the different choices of lift and orientation of the cells.

\subsection{The equivariant intersection form}

Let $W$ be a compact oriented $m$--manifold, $R$ be a commutative ring and $\varphi:\Z[\pi_1(W)] \rightarrow R$ be a morphism. Assume $\partial W$ is decomposed as $\partial W=A\sqcup B$. 
For $q \in \{0,\dots,m\}$, the \emph{twisted intersection form of $W$} with coefficient in $R$, introduced by Reidemeister in \cite{Re39}, is the sesquilinear map
\begin{equation*} \label{equivform}
 \langle\cdot,\cdot\rangle_W^\varphi: H_q^\varphi(W,A) \times H_{m-q}^\varphi(W,B) \longrightarrow R 
\end{equation*}
defined by 
$$ \langle [x \otimes r] , [x' \otimes r'] \rangle_W^\varphi= \sum_{h \in \frac{H_1(W)}{\ker(\varphi)}} \langle x, h.x'\rangle_{\overline W} \, \varphi(h)rr',$$
where here we are abusing notation slightly by letting $\varphi$ denote the group homomorphism from $H_1(W)$ into the multiplicative group $R^\times$, $\overline{W}\twoheadrightarrow W$ is the covering associated with $\ker\varphi$ and $\langle\cdot,\cdot\rangle_{\overline W}$ stands for the algebraic intersection in $\overline{W}$. 
By Blanchfield's duality theorem \cite[Theorem 2.6]{Bl57}, if $W$ is smooth and $\varphi(H_1(W))$ is a free multiplicative subgroup of $R$, this form is non degenerate on $(H_q^\varphi(W,A)/\textrm{Tor}) \times (H_{m-q}^\varphi(W,B)/\textrm{Tor})$.

\section{The twisted absolute homology groups and torsion} \label{sec:AbsHomology}

In this section we derive chain complexes for the twisted homology groups $H^\varphi_*(X)$, assuming that $\partial X\neq \emptyset$. Recall that, in this case, $\partial\Sigma\neq\emptyset$.

\begin{definition}
 Let $V$ be a hyper compression body. A \emph{defining collection of balls} for $V$ is a collection $\Bb$ of $3$--balls properly embedded in $V$ such that $V\setminus\eta(\Bb)$ is a thickening of $\partial_-V$.
\end{definition}

We assume for the remainder of this section that a fixed choice of defining collections $\mathcal{D}_i$ and~$\mathcal{B}_i$ of disks and balls has been made for all $1\leq i\leq n$. 

The following lemmas provide a 3--complex onto which $X$ deformation retracts. Our calculations hinge on a careful understanding of how the cells of this complex are mirrored by simple closed curves on $\Sigma$.

\begin{lemma}\label{lem:CWStructure}
 The manifold $X$ retracts onto $\Sigma\cup\bigcup_{i=1}^n(\Dd_i\cup \Bb_i)$.
\end{lemma}

\begin{proof}
 By definition, $X_i\setminus\eta(\Bb_i)\cong \partial_-X_i\times I$, hence each $X_i$ retracts onto $\partial_+X_i\cup \Bb_i$, hence $X$ retracts onto $Y\cup \bigcup_i\Bb_i$.  Likewise each $C_i$ retracts onto $\partial_+C_i\cup \Dd_i=\Sigma\cup \Dd_i$, so that $Y$ further retracts down to $\Sigma\cup\bigcup_i\Dd_i$.
\end{proof}

\begin{lemma}\label{lem:Retract}
 The quad $(X,Y,\Sigma,*)$ deformation retracts on a CW--complex $(Z_3,Z_2,Z_1,Z_0)$, where $Z_0=*$, $Z_1$ is a bouquet of loops defining a basis of $H_1(\Sigma)$, $Z_2=Z_1\cup\Dd$, and $Z_3=Z_2\cup \Bb$.
\end{lemma}

\begin{proof}
 Retracting as in Lemma~\ref{lem:CWStructure}, we first get $Z_1=\Sigma$ and $Z_2=\Sigma\cup\Dd$. The result follows by further retracting $\Sigma$.
\end{proof}

\begin{corollary} \label{cor:complexX}
 The twisted homology of $X$ is the homology of the following complex.
 \begin{equation}\label{CellularXComplex}\tag{$\Cc'$}
 0\rightarrow H_3^\varphi(X,Y)\rightarrow H_2^\varphi(Y,\Sigma)\rightarrow H_1^\varphi(\Sigma,*)\rightarrow H_0^\varphi(*)
 \end{equation}
 Moreover, if $R$ is a field, there are complex bases $c$ of \ref{CellularXComplex} such that $\tau^\varphi(X;h)=\tau(\Cc';c,h)$ for any homology basis $h$ of $X$ and \ref{CellularXComplex}.
\end{corollary}

\begin{proof}
 Lemma \ref{lem:Retract} shows that the complex above is isomorphic to the cellular homology complex $$0\rightarrow H_3^\varphi (Z_3,Z_2)\rightarrow H_2^\varphi (Z_2,Z_1)\rightarrow H_1^\varphi (Z_1,Z_0)\rightarrow H_0^\varphi(Z_0)$$ via the map induced by the inclusion $Z_3\hookrightarrow X$.  Indeed, since $Z_3$ is a deformation retract of~$X$, it is simply homotopic to~$X$. Hence the assertion on torsion follows by choosing any basis $c$ which is the image of a cellular basis of $Z_3$ under the inclusion map.
\end{proof}

\begin{definition} 
 Let $L_i^\varphi$ denote the submodule of $H_1^\varphi(\Sigma,*)$ generated by the twisted homology classes of the components of $c_i$.
\end{definition}

Following the approach of \cite{FM}, we shall express the complex (\ref{CellularXComplex}) in terms of these submodules. They have the following homological interpretation.

\begin{lemma}
 The module $L_i^\varphi$ naturally identifies with the kernel of the inclusion map $\iota_*:H_1^\varphi(\Sigma)\to H_1^\varphi(C_i)$.
\end{lemma}

\begin{proof}
 Since the components of $c_i$ bound disks in $C_i$, it is clear that $L_i^\varphi\subset \ker(\iota_*)$; since these disks cut $C_i$ into a thickened $\partial_-C_i$, the reverse inclusion follows.
\end{proof}

\begin{lemma}\label{lem:CompressionBody1}
 For all $i$, $H_2^\varphi(C_i,\Sigma)\cong L_i^\varphi$.
\end{lemma}

\begin{proof}
 Since $C_i$ deformation-retracts onto $\Sigma\cup\Dd_i$, we have $H_1^\varphi(C_i,\Sigma)=0$, so that the exact sequence of the pair $(C_i,\Sigma)$ gives $H_2^\varphi(C_i,\Sigma)\cong\ker \big(H_1^\varphi(\Sigma)\to H_1^\varphi(C_i)\big)$. 
\end{proof}

\begin{lemma}\label{lem:CompressionBody2}
 For all $i$, $H_3^\varphi(X_i,C_{i-1}\cup C_i)\cong L_{i-1}^\varphi\cap L_i^\varphi.$
\end{lemma}

\begin{proof}
 Since $X_i$ is a $4$--dimensional $1$--handlebody, its order $2$ and $3$ homology is trivial, and the exact sequence of the pair $(X_i,C_{i-1}\cup C_i)$ gives $H_3^\varphi(X_i,C_{i-1}\cup C_i)\cong H_2^\varphi(C_{i-1}\cup C_i)$. Now the exact sequence of the pair $(C_{i-1}\cup C_i,\Sigma)$ gives 
 $$0\to H_2^\varphi(C_{i-1}\cup C_i)\to H_2^\varphi(C_{i-1},\Sigma)\oplus H_2^\varphi(C_i,\Sigma)\to H_1^\varphi(\Sigma).$$
 Conclude with Lemma~\ref{lem:CompressionBody1}.
\end{proof}

\begin{theorem}\label{thm:XHomology}
 The homology of $X$ is given by the chain complex 
 \begin{equation} \label{complexforX} \tag{$\Cc$}
  0\rightarrow \bigoplus_{i=1}^n(L_{i-1}^\varphi\cap L_i^\varphi)\xrightarrow{\partial_2} \bigoplus_{i=1}^n L_i^\varphi\xrightarrow{\partial_1} H_1^\varphi(\Sigma,*)\xrightarrow{\partial_0} H_0^\varphi(*)
 \end{equation}
 where $\partial_2\big((x_i)_{1\leq i\leq n}\big)=(x_i-x_{i+1})_{1\leq i\leq n}$ and $\partial_1\big((x_i)_{1\leq i\leq n}\big)=\sum_{i=1}^n x_i$.
 Moreover, if $R$ is a field, there are complex bases $c$ of \ref{complexforX} such that $\tau^\varphi(X;h)=\tau(\Cc;c,h)$ for any homology basis $h$ of $X$ and \ref{complexforX}.
\end{theorem}

\begin{proof}
 Since $H_2^\varphi(Y,\Sigma)\cong\oplus_i H_2^\varphi(C_i,\Sigma)$ and $H_3^\varphi(X,Y)\cong\oplus_i H_3^\varphi(X_i,C_{i-1}\cup C_i)$, we can conclude with Corollary~\ref{cor:complexX} and Lemmas~\ref{lem:CompressionBody1} and~\ref{lem:CompressionBody2}.
\end{proof}

\begin{remark}\label{rmk:CellularBasis}
As noted in the proof of Corollary~\ref{cor:complexX}, the bases of $\Cc'$ that can be used to calculate torsions of $X$ are images of cellular bases under the map induced by inclusion $Z_3\hookrightarrow X$. The maps described in Lemmas~\ref{lem:CompressionBody1} and~\ref{lem:CompressionBody2} then define an isomorphism from $\Cc'$ to $\Cc$, and the valid bases of $\Cc$ are then the images of valid bases of $\Cc'$ under this map. This yields the following more concrete description of what valid torsion bases $c$ look like for $\Cc$:
\begin{itemize}
 \item $c_0$ is given by the basepoint $*$,
 \item $c_1$ is defined by any set of loops which form a spine of $\Sigma$,
 \item $c_2$ is any basis corresponding to a tuple of defining curves $(c_i)_{1\leq i\leq n}$,
 \item $c_3$ is any basis corresponding to a tuple of ``double curves" for the pairs $(c_{i-1},c_i)$.
\end{itemize}
By a ``double curve'' for a pair $(c_{i-1},c_i)$, we mean any curve on $\Sigma$ which simultaneously bounds disks in $C_{i-1}$ and $C_i$. The constraints on multisection diagrams imply that $L_{i-1}^\varphi\cap L_i^\varphi$ admits bases represented by double curves, see Figure~\ref{fig:StandardDiagram}.

It might not be easy to find a system of double curves from a diagram, since it implies some possibly unobvious handleslides. However, in practice, it might not be necessary. For instance, the twisting map $\varphi$ often comes from a group morphism $\pi_1(X)\to G$, which induces ring morphisms from $\Z[\pi_1(X)]$ to the group ring $R=\Z[G]$ and to its quotient field $\F=\Q(G)$. One may then choose any bases over $R$ and get a well-defined torsion in $\F^*/R^\times$.
\end{remark}

\begin{corollary}\label{cor:ExplicitHomology}
 We have the following expressions:
 $$H_1^\varphi(X)\cong H_1^\varphi(\Sigma)/(\oplus_i L_i^\varphi),\qquad H_3^\varphi(X)\cong \cap_i L_i^\varphi,$$
 where we slightly abuse notation by viewing $L_i^\varphi\subset H_1^\varphi(\Sigma)\subset H_1^\varphi(\Sigma,*)$.
\end{corollary}

\begin{proof}
 For $H_1$, the pair $(\Sigma,*)$ gives $H_1^\varphi(\Sigma)\cong\ker\big(H_1^\varphi(\Sigma,*)\to H_0^\varphi(*)\big)$.
\end{proof}

A satisfying point in Theorem~\ref{thm:XHomology} is that the modules of the complex \ref{complexforX} are free.

\begin{lemma} \label{lem:FreeModules}
 The modules $H_1^\varphi(\Sigma,*)$ and $L_i^\varphi$ are free $R$--modules of rank $2g+b-1$ and $p$ respectively, where $g$ is the genus of $\Sigma$, $b$ is its number of boundary components and $p$ is the number of curves each collection $c_i$. The modules $L_{i-1}^\varphi\cap L_i^\varphi$ are also free.
\end{lemma}

\begin{proof}
 Since $\partial \Sigma\neq \emptyset$, $\Sigma$ deformation retracts onto a bouquet of $2g+b-1$ loops with central vertex~$*$. Hence $C_1^\varphi(\Sigma,*)\cong R^{2g+b-1}$ is the only non-trivial twisted chain module of $(\Sigma,*)$ and $H_1^\varphi(\Sigma,*)\cong R^{2g+b-1}$. The retraction can be chosen so that the components of $c_i$ are loops of the bouquet, hence $L_i^\varphi$ is a free submodule of $H_1^\varphi(\Sigma,*)$ with basis given by the classes of these components. Morover, up to handleslide, we can assume the components of $c_{i-1}$ and $c_i$ are in standard position (see Figure~\ref{fig:StandardDiagram}), so that a basis of $L_{i-1}^\varphi\cap L_i^\varphi$ is given by the parallel curves in these collections. 
\end{proof}

\section{The twisted relative homology groups and torsion} \label{sec:RelHomology}

In this section, we compute the twisted relative homology and torsion of $X$. The computation of the homology of $(X,\partial X)$ ends up being formally similar to that of $X$, except that minor complication arises from the fact that we need to puncture $X$. 

\begin{definition}\label{def:RDefiningCollection}
 Let $C$ be a lensed compression body. An \emph{$r$--defining collection of disks} for $C$ is a disjoint union $\Dd^r$ of disks, with boundary in $\partial_+C$ or made of an arc in $\partial_+C$ and an arc in $\partial_-C$, such that $C\setminus \eta(\Dd^r)$ is a $3$--ball. The intersection with $\partial_+C$ of an $r$--defining collection of disks for $C$ is a {\em complete collection of arcs and curves} for $C$.
 
 Likewise if $V$ is a hyper compression body then an \emph{$r$--defining collection of balls} for $V$ is a union of $3$--balls $\Bb^r$ such that $V\setminus \eta(\Bb^r)$ is a $4$--ball.
\end{definition}

Such $r$--defining collections of disks do exist. First take a subcollection of a defining collection of disks for $C$, dropping those that do not carry homology relative to boundary. Then add the products with the interval in $\partial_-C\times I$ of arcs that cut $\partial_-C$ into a disjoint union of disks. A similar argument shows existence of $r$--defining collections of balls for the hyper-compression bodies under consideration here.

Fix $r$--defining collections $\Dd_i^r$ and $\Bb_i^r$ of disks and balls for $C_i$ and $X_i$ respectively. Set $\Dd^r=\cup_{i=1}^n\Dd_i^r$ and $\Bb^r=\cup_{i=1}^n\Bb_i^r$. For all $Z\subset X$, let $Z'=Z\setminus \eta(*)$. 

\begin{lemma} \label{lem:RCWStructure}
 The manifold $X'$ deformation retracts onto $\Sigma'\cup\Dd^r\cup\Bb^r\cup\partial X$, where $\Dd^r$ and $\Bb^r$ are subsets of $\Dd^r$ and $\Bb^r$ respectively. 
\end{lemma}

\begin{proof}
 The proof is similar to that of Lemma \ref{lem:CWStructure}, but instead of retracting from the boundary, we retract ``inside-out'' from the puncture $*$. Since $X_i\setminus \eta(\Bb_i^r)$ is a ball and meets $\eta(*)$ in a small $4$--ball that has been ``scooped out'' of the boundary, we obtain a retraction of $X_i'$ onto $(\partial X_i)'\cup \Bb_i^r$.  Carrying this retraction out for each $i$ yields a retraction of $X'$ onto $Y'\cup\Bb^r\cup\partial X$. Since each $C_i\setminus\eta(\Dd_i^r)$ is also a ball which intersects $\eta(*)$ along a scooped out $3$--ball, $Y'$ can further be retracted onto $\Sigma'\cup\Dd^r$.
\end{proof}

\begin{corollary} \label{cor:Rretract}
 The quad $(X',Y'\cup\partial X,\Sigma'\cup\partial X,\partial X)$ deformation retracts rel $\partial X$ onto a CW--complex $(Z_3^\partial\cup\partial X,Z_2^\partial\cup\partial X,Z_1^\partial\cup\partial X,\partial X)$, where $Z_1^\partial$ is made of arcs and loops on $\Sigma'$, $Z_2^\partial=Z_1^\partial\cup\Dd^r$, $Z_3^\partial=Z_2^\partial\cup\Bb^r$.
\end{corollary}

\begin{corollary} \label{cor:RcomplexX}
 The homology of $(X,\partial X)$ is given by the chain complex 
 \begin{equation} \label{complexforXrel} \tag{$\Cb'$}
  H_4^\varphi(X,X')\rightarrow H_3^\varphi(X',Y'\cup \partial X)\rightarrow H_2^\varphi (Y',\Sigma'\cup \partial Y)\rightarrow H_1^\varphi(\Sigma',\partial\Sigma)\to0.
 \end{equation}
 Moreover, if $R$ is a field, then there are complex bases $c$ of \ref{complexforXrel} such that $\tau^\varphi(X,\partial X;h)=\tau(\Cb';c,h)$ for any homology basis $h$ of $(X,\partial X)$ and \ref{complexforXrel}. 
\end{corollary}

\begin{proof}
 Corollary~\ref{cor:Rretract} immediately gives the following complex for $(X',\partial X)$. 
 $$0\to H_3^\varphi(X',Y'\cup \partial X)\rightarrow H_2^\varphi (Y'\cup\partial X,\Sigma'\cup \partial X)\rightarrow H_1^\varphi(\Sigma'\cup \partial X,\partial X)\to 0$$
 The long exact sequence of the triple $(X,X',\partial X)$ provides the left hand side of the complex. Conclude on torsion as in Corollary~\ref{cor:complexX}.
\end{proof}

\begin{definition}
 Let $\J_i^\varphi$ denote the subgroup of $H_1^\varphi(\Sigma',\partial \Sigma)$ generated by any complete collection of arcs and curves for $C_i$ on $\Sigma'$. 
\end{definition}

\begin{lemma}
 The modules $H_1^\varphi(\Sigma',\partial\Sigma)$ and $\J_i^\varphi$ are free $R$--modules of rank $2g+b-1$ and $2g+b-1-p$ respectively. The modules $\J_{i-1}^\varphi\cap\J_i^\varphi$ are also free. Moreover, $\J_i^\varphi$ is the orthogonal complement of $L_i^\varphi$ with respect to the twisted intersection form on $H_1^\varphi(\Sigma,*)\times H_1^\varphi(\Sigma',\partial \Sigma)$.
\end{lemma}

\begin{proof}
 Let $Z_1^\partial$ be any collection of $2g+b-1$ arcs properly embedded in $\Sigma'$, which are pairwise disjoint and cut $\Sigma'$ into a disk. Then $\Sigma'$ retracts onto $Z_1^\partial\cup \partial \Sigma$, showing that $H_1^\varphi(\Sigma',\partial\Sigma)\cong R^{2g+b-1}$. The $1$--complex $Z_1^\partial$ can be chosen so that  $2g-p+b-1$ of the arcs form a complete collection of arcs and curves for~$C_i$, whose twisted homology classes generate $\J_i^\varphi$. A basis of $\J_{i-1}^\varphi\cap\J_i^\varphi$ is provided by a subcollection of these. Since $\J_i^\varphi$ is clearly contained in the orthogonal complement of~$L_i^\varphi$, the equality follows by a dimension argument, using the non-degeneracy of the intersection form.
\end{proof}

\begin{lemma}\label{lem:JCompressionBody1}
 For all $i$, $H_2^\varphi(C_i',(\partial C_i)')\cong \J_i^\varphi$.
\end{lemma}

\begin{proof}
 The long exact sequence of the triple $(C_i',(\partial C_i)',\partial_-C_i)$ and the excision equivalence $((\partial C_i)',\partial_-C_i)\sim(\Sigma',\partial \Sigma)$ give the short exact sequence:
 $$0\to H_2(C_i',(\partial C_i)') \to H_1(\Sigma',\partial\Sigma) \xrightarrow{\zeta} H_1(C_i',\partial_- C_i).$$
 Now $C_i'$ is obtained from a thickened $\partial_- C_i$ by adding only 1--handles, so that the kernel of~$\zeta$ contains the homology classes of curves in $\Sigma'$ that have trivial algebraic intersection with the co-cores of these 1--handles, co-cores whose boundaries generate $L_i^\varphi$. We conclude that $H_2(C_i',(\partial C_i)')=\ker(\zeta)\cong\J_i^\varphi$.
\end{proof}

\begin{lemma}\label{lem:JCompressionBody2}
 For all $i$, $H_3^\varphi(X_i',(\partial X_i)')\cong \J_{i-1}^\varphi\cap \J_i^\varphi$. 
\end{lemma}

\begin{proof}
 Since $X_i'$ is obtained from a thickened $\partial_-X_i$ by adding $1$--handles, the exact sequence of the triple $(X_i',(\partial X_i)',\partial_-X_i)$ gives an isomorphism $H_3^\varphi(X_i',(\partial X_i)')\cong H_2^\varphi((\partial X_i)',\partial_-X_i)$, and this last module is isomorphic to $H_2(C_{i-1}'\cup C_i',\partial_-C_{i-1}\cup\partial_-C_i)$. Now the long exact sequence of the triple $(C_{i-1}'\cup C_i',(\partial C_{i-1})'\cup (\partial C_i)',\partial_-C_{i-1}\cup\partial_-C_i)$ shows that $$H_2^\varphi(C_{i-1}'\cup C_i',\partial_-C_{i-1}\cup\partial_-C_i)\cong\ker\{H_2^\varphi(C_{i-1}',(\partial C_{i-1})')\oplus H_2^\varphi(C_i',(\partial C_i)')\rightarrow H_1^\varphi(\Sigma',\partial\Sigma)\}$$ 
 and the conclusion follows from Lemma \ref{lem:JCompressionBody1}.
\end{proof}

\begin{theorem}\label{thm:RelXHomology}
 The homology of $(X,\partial X)$ is given by the chain complex 
 \begin{equation} \label{complexXrel} \tag{$\Cb$}
  H_2^\varphi(\Sigma,\Sigma')\xrightarrow{\partial_3} \bigoplus_i\J_{i-1}^\varphi\cap \J_{i}^\varphi\xrightarrow{\partial_2} \bigoplus_i\J_i^\varphi\xrightarrow{\partial_1} H_1^\varphi(\Sigma',\partial\Sigma)\rightarrow 0
 \end{equation}
 where $\partial_3([\Sigma])=[\partial(\Sigma\setminus\Sigma')]$, $\partial_2((x_i)_{1\leq i\leq n})=((x_i-x_{i+1})_{1\leq i\leq n})$ and $\partial_1((x_i)_{1\leq i\leq n})=\sum_{i=1}^nx_i$.
 Moreover, if $R$ is a field, there is a complex basis $c$ of \ref{complexXrel} such that $\tau^\varphi(X,\partial X;h)=\tau(\Cb;c,h)$ for any homology basis~$h$ of $(X,\partial X)$ and \ref{complexXrel}.
\end{theorem}

\begin{proof}
 Start with the complex (\ref{complexforXrel}) of Corollary~\ref{cor:RcomplexX}. The order 2 and 3 terms are given by Lemmas~\ref{lem:JCompressionBody1} and~\ref{lem:JCompressionBody2}. A generator of $H_4^\varphi(X,X')$ is sent onto the class of $\partial\eta(*)$ in $H_3^\varphi(X',Y'\cup\partial X)$. Following the isomorphisms in Lemma~\ref{lem:JCompressionBody2}, we see that this class corresponds to the class in $\bigoplus_i\J_{i-1}^\varphi\cap \J_i^\varphi$ of the curve $\partial\eta(*)$, where the neighborhood is now understood in $\Sigma$, which is the boundary of a generator of $H_2(\Sigma,\Sigma')$. 
\end{proof}

\begin{remark}

As with the absolute case, we can obtain a concrete description of what valid torsion bases $c$ look like for $\Cc_\partial$:

\begin{itemize}

\item $c_1=\{ [e_1],[e_2],\ldots , [e_n]\}$, where each $e_i$ is an edge of $Z_1^\partial$ ({\em i.e.} any set of arcs which cut $\Sigma$ into a disk),

\item $c_2=$ any basis corresponding to a tuple of complete collections of arcs and curves for $C_i$,

\item $c_3=$ any basis corresponding to a tuple of ``double arcs and curves" for the pairs $(C_{i-1},C_i)$,

\item $c_4=$ the fundamental class of $H_2(\Sigma,\Sigma')$.

\end{itemize}
\end{remark}

\begin{corollary}\label{cor:ExplicitRelHomology}
 We have the following expressions for the twisted homology of $(X,\partial X)$:
 $$H_1^\varphi(X,\partial X)\cong H_1^\varphi(\Sigma',\partial \Sigma)/\oplus_i\J_i^\varphi, \quad 
 H_3^\varphi(X,\partial X)\cong \cap_i\overline{\J}_i^\varphi,$$
 where $\overline{\J}_i^\varphi$ denotes the image of $\J_i^\varphi$ under the inclusion map $H_1^\varphi(\Sigma',\partial \Sigma)\rightarrow H_1^\varphi(\Sigma,\partial\Sigma)$. 
\end{corollary}

\begin{proof}
 For $H_3$, the long exact sequence of the triple $(\Sigma,\Sigma',\partial \Sigma)$ gives the exact sequence $$H_2^\varphi(\Sigma,\Sigma')\xrightarrow{\zeta} H_1^\varphi(\Sigma',\partial \Sigma)\to H_1^\varphi(\Sigma,\partial\Sigma)\to0$$ and we have $H_3^\varphi(X,\partial X)\cong(\cap_i\J_i^\varphi)/\mathrm{Im}(\zeta)$.
\end{proof}

\section{Intersection forms} \label{sec:IntForm}

We keep in this section the assumption that $\partial X\neq\emptyset$. The intersection forms are formally identical to the closed case treated in \cite{FM}. The upshot is that the intersections between various cycles in $X$ can all be made to coincide with intersections in~$\Sigma$.  Below we assume that $\Sigma'=\Sigma\setminus\eta(*)$, so that there is a natural isomorphism $H_1^\varphi(\Sigma,*)\cong H_1^\varphi(\Sigma',\partial\eta(*))$, and we identify each $L_i^\varphi$  with its image under this map below. 

\begin{theorem} \label{thm:IntersectionForm}
 Suppose $h_1=[(x_i)_{1\leq n}]\in H_2^\varphi(X)$ and $h_2=[(y_i)_{1\leq n}]\in H_2^\varphi(X,\partial X)$, where $(x_i)_{1\leq n}\in \oplus_i L_i^\varphi$, and $(y_i)_{1\leq n}\in \oplus_i\J_i^\varphi$. Then 
 \[\langle h_1,h_2\rangle_X^\varphi=\sum_{1\leq i<j\leq n}\langle x_i,y_j\rangle_\Sigma^\varphi\] 
 where here $\langle\cdot,\cdot\rangle_{X}^\varphi$ and $\langle\cdot,\cdot\rangle_\Sigma^\varphi$ are the equivariant intersection forms on $H_2^\varphi(X)\times H_2^\varphi(X,\partial X)$ and  $H_1^\varphi(\Sigma,*)\times H_1^\varphi(\Sigma',\partial \Sigma)$ respectively.
\end{theorem}

\begin{proof}
 It suffices to show that the analogous claim holds true in the untwisted integral homology groups of $\overline{X}$, which denotes the cover of $X$ associated to $\ker(\varphi)$.  For any $Z\subset X$ let $\overline{Z}$ denote the inverse image of $Z$ under the cover $\overline{X}\rightarrow X$.  Since $\pi_1(\Sigma)$, $\pi_1(C_i)$, and $\pi_1(X_i)$ all surject onto $\pi_1(X)$ via the inclusion map, $\overline{\Sigma}$, $\overline{C_i}$, and $\overline{X_i}$ are connected as well. In the finite case these lifts combine to form a multisection of $\overline{X}$, and in the case of an infinite sheeted cover they form what is essentially a multisection, except the pieces involved have infinite genus.  In particular, just as in the finite case, $\eta(\overline{\Sigma})$ is a trivial disk bundle and the lifted compression bodies $\overline{C_i}$ meet each disk in the bundle along rays which are disjoint except at the center point.

There is a cellular structure on $\overline{X}$ obtained by lifting the cell structures of $X$ and $(X,\partial X)$ described in Lemmas~\ref{lem:CWStructure} and~\ref{lem:RCWStructure} to $\overline{X}$. If $Z_2$ is the $2$--skeleton of $X$ described in Lemma~\ref{lem:CWStructure}, then $\overline{Z_2}$ is a $2$-skeleton for $\overline{X}$ which lies in $\overline{\cup_i C_i}$. Similarly, let $Z_2^\partial$ be the relative $2$--skeleton of Lemma~\ref{lem:RCWStructure}. As observed in \cite{FM}, we may push each $\overline{Z_2^\partial}\cap \overline{C_i}$ slightly into its collar so that it is pushed into $\overline{\cup_{1\leq j\leq i}X_j}$. This being done, the intersections between $2$--chains in $C_2(\overline{Z_2})$ and $C_2(\overline{Z_2^\partial})$ will coincide with intersections between the boundaries of the sub-chains lying just in $\overline{Z_2\cap C_i}$ with the boundaries of the sub-chains lying just in $\overline{Z_2^\partial\cap(\cup_{i<j\leq n} C_j)}$, see the left-hand side of Figure~\ref{fig:IntForm}. 
\end{proof}

\begin{figure}[htb]
\begin{center}
\begin{tikzpicture} [scale=0.5]
\begin{scope}
 \draw (0,0) circle (5);
 \foreach \s in{1,...,5,6} {
 \draw[rotate=60*(\s+1)] (0,0) -- (5,0);
 \draw[rotate=60*(\s+1)] (5.8,0) node {$C_\s$};}
 \draw[rounded corners=5pt,green] 
 (0,1) .. controls +(-2,-2) and +(1,-1) .. (0.4,0.3) -- (2.7,4.2)
 (0,1) .. controls +(-2.5,-1.5) and +(0,-1) .. (0.9,-0.3) -- (5,-0.3)
 (0,1) .. controls +(-2.5,-0.85) and +(-1,-0.5) .. (0.4,-1) -- (2.3,-4.45)
 (0,1) .. controls +(-2.5,-0.4) and +(-1,0.5) .. (-0.9,-1.1) -- (-2.7,-4.2)
 (0,1) .. controls +(-1,0) and +(0,0.5) .. (-2,0.3) -- (-5,0.3)
 (0,1) -- (-0.4,1.1) -- (-2.3,4.45);
\end{scope}
\begin{scope} [xshift=15cm]
 \draw (0,0) circle (5);
 \foreach \s in{1,...,5,6} {
 \draw[rotate=60*(\s+1)] (0,0) -- (5,0);
 \draw[rotate=60*(\s+1)] (5.8,0) node {$C_\s$};}
 \draw[green] 
 (2.3,4.45) -- (0,1) -- (-2.3,4.45)
 (5,0.3) -- (0,1) -- (-5,0.3)
 (2.7,-4.2) -- (0,1) -- (-2.7,-4.2);
\end{scope}
\end{tikzpicture}
\end{center}
\caption{Pushing the relative $2$--skeleton}
\label{fig:IntForm}
\end{figure}
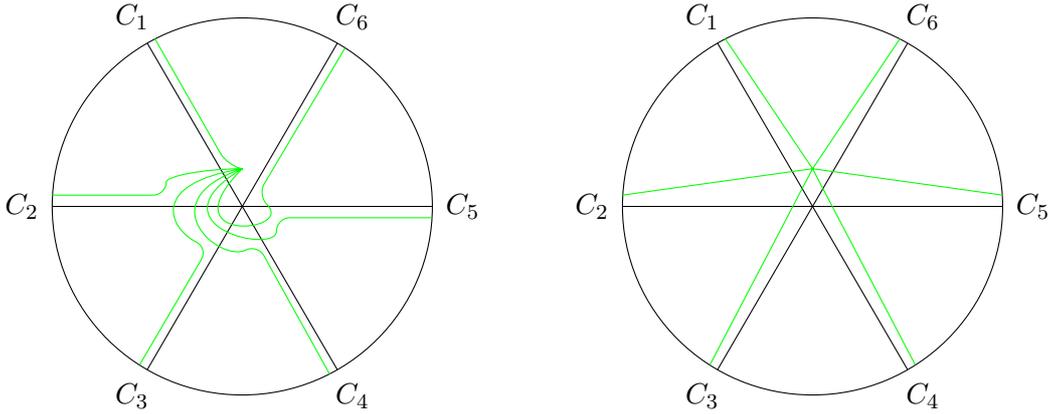

\begin{remark}
 Different expressions can be given for the intersection form by diversely pushing the relative $2$--skeleton. The right-hand side of Figure~\ref{fig:IntForm} suggests another possibility with less terms.
\end{remark}

The intersection forms on the odd-dimensional homology groups are described even more simply.

\begin{theorem} \label{thIntFormOdd}
 Suppose that either $h_1\in H_1^\varphi(X)$ corresponds to the element $a\in H_1^\varphi(\Sigma,*)$ and $h_2\in H_3^\varphi(X,\partial X)$ corresponds to the element $b\in \cap_i\J_i^\varphi$, or $h_1\in H_1^\varphi(X,\partial X)$ corresponds to the element $a\in H_1^\varphi(\Sigma',\partial \Sigma)$ and $h_2\in H_3^\varphi(X)$ corresponds to the element $b\in \cap_i L_i^\varphi$. Then \[\langle h_1,h_2\rangle_X^\varphi=\langle a,b\rangle_\Sigma^\varphi.\]
\end{theorem}

\begin{proof}
 The proof is similar in structure to the proof of Theorem \ref{thm:IntersectionForm}, except that now we observe that every chain in $H_1(\overline{X})$ or $H_1(\overline{X},\partial\overline{X})$ is geometrically represented by linear combinations of curves $c\subset \overline{\Sigma}\subset \overline{X}$, and the chains in $H_1(\overline{X},\partial\overline{X})$ or $H_3(\overline{X})$ can be geometrically represented by linear combinations of balls which meet $\overline{\Sigma}$ only in linear combinations of double curves.  Thus no isotopy is needed, the intersections between the 1--chains and the 3--chains already correspond exactly to the intersections of their representatives in $H_1(\overline{\Sigma})$.  
\end{proof}

\section{The case of closed $4$--manifolds} \label{sec:HomologyClosed}

In this section, we compute the twisted homology, torsion and intersection forms when $X$ is closed. It mainly follows the lines of the computation of relative homology, since we need again to puncture $X$. However, it mixes some features of the absolute and relative cases. 
For instance, when $X$ is closed, $r$--defining collections of disks and balls are the same as ordinary defining collections. 
Since there is no additive difficulty with regards to the non-closed case, we skip the details.

We fix $\star\in\Sigma$; for $Z\subset X$, we set $Z'=Z\setminus\eta(\star)$ and we fix $*\in\partial\Sigma'$. Let $\Dd$ and $\Bb$ be unions of defining collections of disks and balls for the $C_i$ and the $X_i$ respectively.
Lemma~\ref{lem:RCWStructure} still holds, and provides the following corollary.
\begin{lemma} \label{lem:complexClosedX}
 The quad $(X',Y',\Sigma',*)$ deformation retracts onto a CW--complex $(Z_3,Z_2,Z_1,Z_0)$, where $Z_0=*$, $Z_1$ is made of loops on $\Sigma'$, $Z_2=Z_1^\partial\cup\Dd$, $Z_3=Z_2\cup\Bb$. 
 Subsequently, the homology of $X$ is given by the chain complex 
 \begin{equation} \label{complexforXclosed} \tag{$\Ccl'$}
  H_4^\varphi(X,X')\rightarrow H_3^\varphi(X',Y')\rightarrow H_2^\varphi (Y',\Sigma')\rightarrow H_1^\varphi(\Sigma',*)\to H_0^\varphi(*).
 \end{equation}
 Moreover, if $R$ is a field, then there are complex bases $c$ of \ref{complexforXclosed} such that $\tau^\varphi(X;h)=\tau(\Ccl';c,h)$ for any homology basis $h$ of $X$ and \ref{complexforXclosed}. 
\end{lemma}

Now, $L_i^\varphi$ denotes the submodule of $H_1^\varphi(\Sigma',*)$ generated by the homology classes of the curves in $c_i$.

\begin{lemma}
 The modules $H_1^\varphi(\Sigma',*)$ and $L_i^\varphi$ are free $R$--modules of rank $2g$ and $g$ respectively. The modules $L_{i-1}^\varphi\cap L_i^\varphi$ are also free. Moreover, $L_i^\varphi$ is a lagrangian for the twisted intersection form on $H_1^\varphi(\Sigma',*)$.
\end{lemma}

\begin{lemma}\label{lem:HandleBody}
 For all $i$, $H_2^\varphi(C_i',(\partial C_i)')\cong L_i^\varphi$ and $H_3^\varphi(X_i',(\partial X_i)')\cong L_{i-1}^\varphi\cap L_i^\varphi$. 
\end{lemma}

\begin{theorem} \label{thm:XClosedHomology}
 If $X$ is closed, the twisted homology of $X$ is given by the chain complex
 \begin{equation} \label{complexXclosed} \tag{$\Ccl$}
  H_2^\varphi(\Sigma,\Sigma')\xrightarrow{\partial_3} \bigoplus_i(L_{i-1}^\varphi\cap L_{i}^\varphi)\xrightarrow{\partial_2} \bigoplus_i L_i^\varphi\xrightarrow{\partial_1} H_1^\varphi(\Sigma',*)\rightarrow H_0^\varphi(*)
 \end{equation}
 where $\partial_3([\Sigma])=[\partial\Sigma']$, $\partial_2((x_i)_{1\leq i\leq n})=((x_i-x_{i+1})_{1\leq i\leq n})$ and $\partial_1((x_i)_{1\leq i\leq n})=\sum_{i=1}^nx_i$.
 Moreover, if $R$ is a field, there is a complex basis $c$ of \ref{complexXclosed} such that $\tau^\varphi(X;h)=\tau(\Ccl;c,h)$ for any homology basis~$h$ of~$X$ and \ref{complexXclosed}. 
\end{theorem}

\begin{remark}

Once again, we can describe valid torsion bases $c$ for $\Ccl$:

\begin{itemize}

\item $c_0=[*]$,

\item $c_1=$ any set of loops based at $*$ which cut $\Sigma$ into a disk,

\item $c_2=$ any basis corresponding to a tuple of defining collections of curves for $C_i$,

\item $c_3=$ any basis corresponding to a tuple of ``double curves" for the pairs $(C_{i-1},C_i)$,

\item $c_4=$ the fundamental class of $H_2(\Sigma,\Sigma')$.

\end{itemize}

\end{remark}

Expressions for the intersection forms on $H_2^\varphi(X)$ and on $H_1^\varphi(X)\times H_3^\varphi(X)$ are again obtained in terms of the intersection form on $H_1^\varphi(\Sigma',*)$. Stricly speaking, this intersection form is defined on $H_1^\varphi(\Sigma',*_1)\times H_1^\varphi(\Sigma',*_2)$ for two distinct basepoints $*_1$ and $*_2$ on $\partial\Sigma'$. 

\begin{theorem} \label{thm:IntFormClosed}
 Suppose $h_1=[(x_i)_{1\leq n}], h_2=[(y_i)_{1\leq n}]\in H_2^\varphi(X)$, where $(x_i)_{1\leq n}, (y_i)_{1\leq n}\in \oplus_i L_i^\varphi$. Then 
 \[\langle h_1,h_2\rangle_X^\varphi=\sum_{1\leq i<j\leq n}\langle x_i,y_j\rangle_\Sigma^\varphi\] 
 where here $\langle\cdot,\cdot\rangle_{X}^\varphi$ and $\langle\cdot,\cdot\rangle_\Sigma^\varphi$ are the equivariant intersection forms on $H_2^\varphi(X)$ and $H_1(\Sigma',*)$.
 
 Suppose that $h_1\in H_1^\varphi(X)$ corresponds to the element $a\in H_1^\varphi(\Sigma',*)$ and $h_2\in H_3^\varphi(X)$ corresponds to the element $b\in \cap_i L_i^\varphi$. Then $\displaystyle\langle h_1,h_2\rangle_X^\varphi=\langle a,b\rangle_\Sigma^\varphi$.
\end{theorem}

\section{The boundary: monodromy and homology} \label{sec:Boundary}

In this section, we assume $\partial X\neq\emptyset$ and we compute the action of the monodromy of the open book induced on $\partial X$ on the homology of the page $\Sigma_\partial$; we then deduce the homology of $\partial X$. All homology groups are considered with coefficients in $\Z$. We denote $\Sigma_i$ the result of compressing $\Sigma$ along~$c_i$, which is a copy of $\Sigma_\partial$. Given a compact surface $S$ with no closed component, a {\em cut system for $S$} is a family of arcs on $S$ that cuts each component of $S$ into a disk. 

Our main tool is the algorithm of Castro, Gay and Pinz\'on-Caicedo which describes the monodromy of the open book from a trisection diagram \cite{CGPC1}. Although they work with trisections in the case of a connected page, their result extends directly to the setting of multisections with multiple boundary components.

\begin{proposition}[Castro--Gay--Pinz\'on-Caicedo] \label{prop:AlgoCGPC}
Let $e$ be any choice of arcs in $\Sigma$, disjoint from~$c_1$, that forms a cut system for $\Sigma_1$. The monodromy $\phi:\Sigma_1\rightarrow \Sigma_1$ which defines the open book decomposition of $\partial X$ is encoded by its action on $e$, which in turn is described by the following algorithm. 
For $i$ running from $1$ to $n$, slide the curves $c_{i+1}$ over one another and slide the arcs $e$ over the curves $c_i$, until $e$ is disjoint from the curves $c_{i+1}$.
The family $c_1'\cup e'$ which results from these $n$ steps will generally be distinct from the original family $c_1\cup e$.  Perform one final sequence of handleslides of the arcs and curves $c_1'\cup e'$ which sends $c_1'$ to $c_1$. The resulting cut system $e'$ for $\Sigma_1$ is $\phi(e)$.

It is necessary to explicitly index the arcs $e$ and keep track of this index throughout the algorithm, but the simple closed curves $c_i$ need not be indexed.
\end{proposition}

We denote $L_i$ (resp. $\L_i$) the subgroup of $H_1(\Sigma)$ (resp. $H_1(\Sigma,\partial\Sigma)$) generated by the homology classes of the curves in $c_i$, and we let $\J_i$ (resp. $J_i$) denote its orthogonal in $H_1(\Sigma,\partial\Sigma)$ (resp. $H_1(\Sigma)$). 

\begin{lemma}
 The groups $L_i$, $J_i$, $\L_i$ and $\J_i$ are free abelian groups of rank $p$, $g+h+b-1$, $g-h$ and $2g+b-p-1$ respectively, where $g$ is the genus of $\Sigma$, $h$ the genus of $\Sigma_\partial$ and $b$ the number of boundary components of both. Moreover, $L_i$ and $\L_i$ are primitive subbroups of $J_i$ and $\J_i$  respectively, so that the quotients $J_i/L_i$ and $\J_i/\L_i$ both are free abelian groups of rank $g+h+b-p-1$.
\end{lemma}
\begin{proof}
 Everything can be read from the surface $\Sigma$ with the family $c_i$ in standard position, see Figure~\ref{fig:CompressionBody}.
\end{proof}

\begin{figure}[htb]
\begin{center}
\begin{tikzpicture} [scale=0.3]
\newcommand{\trou}{
(2,0) ..controls +(0.5,-0.25) and +(-0.5,-0.25) .. (4,0)
(2.3,-0.1) ..controls +(0.6,0.2) and +(-0.6,0.2) .. (3.7,-0.1)}
\draw (0,0) ..controls +(0,1) and +(-2,1) .. (4,2);
\draw (4,2) ..controls +(1,-0.5) and +(-1,0) .. (6,1.25);
\draw[dashed] (6,1.25);
\draw (0,0) ..controls +(0,-1) and +(-2,-1) .. (4,-2);
\draw (4,-2) ..controls +(1,0.5) and +(-1,0) .. (6,-1.25);
\foreach \x/\y in {6/0,18/0,38/3} {
\begin{scope} [xshift=\x cm,yshift=\y cm]
\draw (0,1.25) ..controls +(1,0) and +(-2,1) .. (4,2);
\draw (4,2) ..controls +(2,-1) and +(-2,-1) .. (8,2);
\draw (8,2) ..controls +(2,1) and +(-1.2,0) .. (12,1.25);
\draw (0,-1.25) ..controls +(1,0) and +(-2,-1) .. (4,-2);
\draw (4,-2) ..controls +(2,1) and +(-2,1) .. (8,-2);
\draw (8,-2) ..controls +(2,-1) and +(-1.2,0) .. (12,-1.25);
\end{scope}}
\foreach \x in {0,6,12,18,24} {
\draw[xshift=\x cm] \trou;}
\foreach \x in {-0.3,5.7,11.7,17.7,23.7} {
\draw[color=red,xshift=\x cm] (3,-0.2) ..controls +(0.2,-0.5) and +(0.2,0.5) .. (3,-2.3);
\draw[dashed,color=red,xshift=\x cm] (3,-0.2) ..controls +(-0.2,-0.5) and +(-0.2,0.5) .. (3,-2.3);}
\foreach \x in {0,6,12,18,24} {
\draw[color=aqua,xshift=\x cm] (3,-0.2) ..controls +(0.2,-0.5) and +(0.2,0.5) .. (3,-2.3);
\draw[dashed,color=aqua,xshift=\x cm] (3,-0.2) ..controls +(-0.2,-0.5) and +(-0.2,0.5) .. (3,-2.3);}
\foreach \x in {0.3,6.3,12.3,18.3,24.3} {
\draw[color=bviolet,xshift=\x cm] (3,-0.2) ..controls +(0.2,-0.5) and +(0.2,0.5) .. (3,-2.3);
\draw[dashed,color=bviolet,xshift=\x cm] (3,-0.2) ..controls +(-0.2,-0.5) and +(-0.2,0.5) .. (3,-2.3);}
\foreach \x/\y/\c in {37.9/-3/aqua,38.2/-3/bviolet,37.9/3/aqua,38.2/3/bviolet,43.9/3/aqua,44.2/3/bviolet,37.9/9/aqua,38.2/9/bviolet} {
\draw[color=\c,xshift=\x cm,yshift=\y cm] (3,-0.2) ..controls +(0.2,-0.5) and +(0.2,0.5) .. (3,-2.3);
\draw[dashed,color=\c,xshift=\x cm,yshift=\y cm] (3,-0.2) ..controls +(-0.2,-0.5) and +(-0.2,0.5) .. (3,-2.3);}
\foreach \x/\y in {38/-3,38/3,44/3,38/9} {
\draw[color=aqua,xshift=\x cm,yshift=\y cm] (3,0)ellipse(1.5 and 0.7);
\draw[color=bviolet,xshift=\x cm,yshift=\y cm] (3,0)ellipse(1.8 and 1);}
\foreach \y in {-9,9} {
\draw[color=bviolet,xshift=44cm,yshift=\y cm] (0,0.8) arc (90:270:0.8);}
\foreach \x/\y in {38/3,44/3,38/-3,38/9} {
\draw[xshift=\x cm,yshift=\y cm] \trou;}
\foreach \x/\y in {44/-3,50/3}{
\begin{scope} [xshift=\x cm,yshift=\y cm]
\draw (0,1.25) ..controls +(0.4,-1) and +(0.4,1) .. (0,-1.25);
\draw[dashed] (0,1.25) ..controls +(-0.4,-1) and +(-0.4,1) .. (0,-1.25);
\end{scope}}
\begin{scope} [xshift=26cm,yshift=-3cm]
\draw (14,2) ..controls +(2,1) and +(-1.2,0) .. (18,1.25);
\draw (14,-2) ..controls +(2,-1) and +(-1.2,0) .. (18,-1.25);
\draw (12,1.25) ..controls +(0.5,0) and +(-1,-0.5) .. (14,2);
\draw (12,-1.25) ..controls +(0.5,0) and +(-1,0.5) .. (14,-2);
\end{scope}
\begin{scope} [xshift=26cm,yshift=9cm]
\draw (14,2) ..controls +(2,1) and +(-1,0) .. (18,2);
\draw (14,-2) ..controls +(2,-1) and +(-1,0) .. (18,-2);
\draw (12,1.25) ..controls +(0.5,0) and +(-1,-0.5) .. (14,2);
\draw (12,-1.25) ..controls +(0.5,0) and +(-1,0.5) .. (14,-2);
\end{scope}
\begin{scope} [xshift=26cm,yshift=-9cm]
\draw (12,1.25) ..controls +(2,0) and +(-2,0) .. (18,2);
\draw (12,-1.25) ..controls +(2,0) and +(-2,0) .. (18,-2);
\end{scope}
\foreach \y in {9,-9} {
\begin{scope} [xshift=44cm,yshift=\y cm]
\foreach \s in {1,-1} {
\draw (0,2*\s) ..controls +(0.2,-0.5*\s) and +(0.2,0.5*\s) .. (0,0.5*\s);
\draw[dashed] (0,2*\s) ..controls +(-0.2,-0.5*\s) and +(-0.2,0.5*\s) .. (0,0.5*\s);}
\draw[color=aqua] (-0.3,2) ..controls +(0.2,-0.5) and +(0.2,0.5) .. (-0.3,0.4);
\draw[color=aqua,dashed] (-0.3,2) ..controls +(-0.2,-0.5) and +(-0.2,0.5) .. (-0.3,0.4);
\draw (0,0.5) arc (90:270:0.5);
\draw (0,0.5) arc (90:270:0.5);
\end{scope}}
\foreach \y in {-6,0,6} {
\draw (38,1.75+\y) arc (90:270:1.75);}
\foreach \s in {1,-1} {
\draw (30,1.25*\s) .. controls +(3,0) and +(-5,0) .. (38,10.25*\s);}
\foreach \y in {-9,-3,3} {
\foreach \x/\c in {37.9/red,38.2/aqua} {
\begin{scope} [xshift=\x cm,yshift=\y cm,\c]
\draw (0,1.25) ..controls +(0.4,-1) and +(0.4,1) .. (0,-1.25);
\draw[dashed] (0,1.25) ..controls +(-0.4,-1) and +(-0.4,1) .. (0,-1.25);
\end{scope}}}
\end{tikzpicture}
\end{center} \caption{Curves on $\Sigma$ for the compression body $C_i$\\[2pt]{\footnotesize The curves of $c_i$ are in red.\\ The homology classes of the blue and violet curves form bases of $J_i$ and $\J_i$ respectively.}} \label{fig:CompressionBody}
\end{figure}
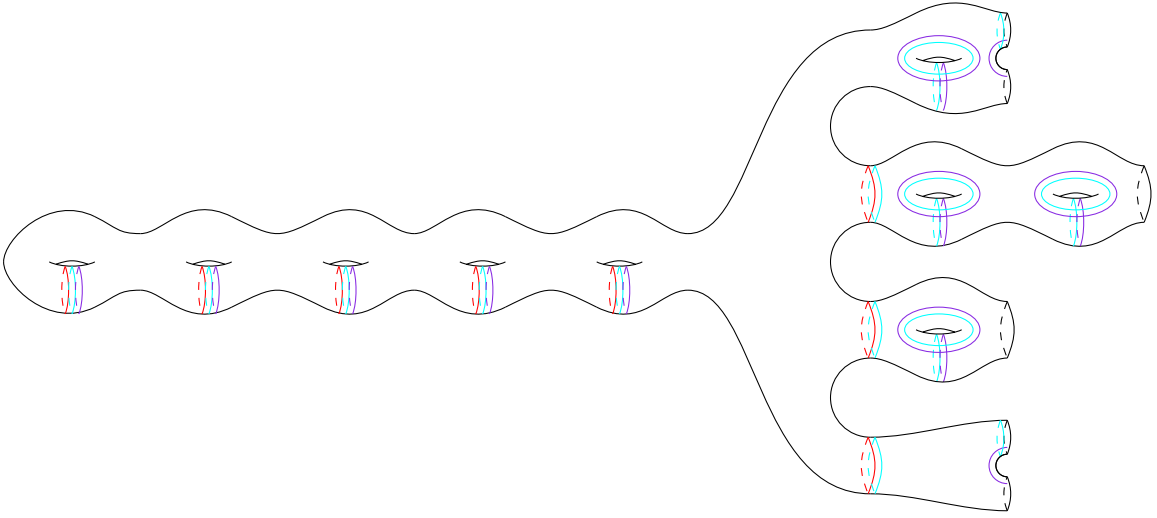

\begin{lemma}
 There are natural identifications $\displaystyle H_1(\Sigma_i,\partial\Sigma)\cong\frac{\J_i}{\L_i}$ and $\displaystyle H_1(\Sigma_i)\cong\frac{J_i}{L_i}$.
\end{lemma}

\begin{proof}
 We first wiew $C_i$ as a thickened $\Sigma_i$ with $1$--handles attached on the positive boundary, whose co-cores are the curves in $c_i$. This shows that $H_2(C_i,\Sigma_i)=0$ and $H_1(C_i,\Sigma_i)$ is generated by the classes of the cores of the $1$--handles. Hence the long exact sequence of the triple $(C_i,\partial C_i,\Sigma_i)$ gives:
 $$0\to H_2(C_i,\partial C_i) \to H_1(\Sigma,\partial\Sigma) \to H_1(C_i,\Sigma_i) \to0.$$
 Now the image of an element of $H_1(\Sigma,\partial \Sigma)$ is determined by its algebraic intersection with the curves in $c_i$, thus $H_2(C_i,\partial C_i)\cong\J_i$.  
 
 Likewise, viewing the compression body $C_i$ as a thickened $\Sigma$ with $2$--handles glued along the curves in $c_i$ on the negative boundary shows that $H_1(C_i,\Sigma)=0$ and $H_2(C_i,\Sigma)$ is generated by the classes of the cores of the $2$--handles. Now the long exact sequence of the triple $(C_i,\partial C_i,\Sigma)$ gives:
 $$H_2(C_i,\Sigma)\to H_2(C_i,\partial C_i) \to H_1(\Sigma_i,\partial\Sigma)\to0.$$
 Since the $2$--handles are glued along the curves in $c_i$, $H_2(C_i,\Sigma)$ is sent onto $\L_i$ in \mbox{$H_2(C_i,\partial C_i)\cong \J_i$.}
 
 Replacing $\partial C_i$ by $\Sigma\sqcup\Sigma_i$, the first step gives $0\to H_2(C_i,\Sigma\sqcup\Sigma_i) \to H_1(\Sigma) \to H_1(C_i,\Sigma_i) \to0$ and $H_2(C_i,\Sigma\sqcup\Sigma_i)\cong J_i$, and the second step gives $H_2(C_i,\Sigma)\to H_2(C_i,\Sigma\sqcup\Sigma_i) \to H_1(\Sigma_i)\to0$ and $H_2(C_i,\Sigma)\cong L_i$ in $H_2(C_i,\Sigma\sqcup\Sigma_i)\cong J_i$. 
\end{proof}

Let $e$ be a family of arcs in $\Sigma$, disjoint from $c_1$, that forms a cut system for~$\Sigma_1$; note that it defines a basis of $H_1(\Sigma_1,\partial\Sigma)$. Let $a_i$ be a family of simple closed curves on~$\Sigma$ that defines a basis of $\frac{\L_i}{\L_i\cap \L_{i+1}}$ or $\frac{L_i}{L_i\cap L_{i-1}}$ (in the sequel, we may consider their homology classes in $H_1(\Sigma)$ or $H_1(\Sigma,\partial\Sigma)$).
For $\mu=(\mu_i)_{1\leq i\leq s}$ and $\mu'=(\mu'_i)_{1\leq i\leq t}$ families of $H_1(\Sigma,\partial\Sigma)$ and $H_1(\Sigma)$, define the matrix $\mu\cdot\mu'=\left(\langle\mu_i,\mu'_j\rangle_\Sigma\right)_{1\leq i\leq s, 1\leq j\leq t}$. 
\begin{proposition} \label{prop:monodromy}
 Let $\phi:\Sigma_1 \rightarrow \Sigma_1$ be the monodromy which defines the open book on $\partial X$. Define matrices $R_i$ and families $e_i$ in $H_1(\Sigma,\partial\Sigma)$ recursively as follows:
 \begin{itemize}
  \item $R_0=0$ and $e_1=e$,
  \item $R_i=-(e_i\cdot a_{i+1})(a_i\cdot a_{i+1})^{-1}$ and $e_{i+1}=e_i+R_ia_i$.
 \end{itemize}
 Fix a basis of the free $\Z$--module $\J_1$ which admits $e$ as a subfamily and write the families $e$ and $e_{n+1}$ in this basis. Then the action of the monodromy of the open book of $\partial X$ on $H_1(\Sigma_1,\partial\Sigma)\cong\frac{\J_1}{\L_1}$ is given in the basis $e$ by the matrix of  $R=e^te_{n+1},$ where $e^t$ is the transpose of $e$. 
\end{proposition}
\begin{proof}
 Following the algorithm of Proposition~\ref{prop:AlgoCGPC}, we define families of arcs and curves $e_i$ on $\Sigma$, disjoint from $c_i$, that define bases of $H_1(\Sigma_i,\partial\Sigma)$, by $e_1=e$ and $e_{i+1}=e_i+r_ia_i$, where the $r_i$ are matrices to compute. Since $e_i$ is disjoint from $c_i$, we have
 $0=e_{i+1}\cdot a_{i+1}=e_i\cdot a_{i+1}+r_i(a_i\cdot a_{i+1}),$
 so that $r_i=R_i$. Now $e_{n+1}$ expresses $\phi(e)$ in the fixed basis of $\J_1$. Multiply by $e^t$ to get it in the basis $e$ of $H_1(\Sigma_1,\partial\Sigma)$.
\end{proof}

The following lemma gives the homology of a 3--manifold from an open book decomposition. 

\begin{lemma}
 Let $M$ be a 3--manifold with an open book $(S,\phi)$. The homology of $M$ is the homology of the complex:
 $$0 \to \Z^s \xrightarrow{\scriptstyle{0}} H_1(S,\partial S) \xrightarrow{\xi} H_1(S)\xrightarrow{\scriptstyle{0}} \Z^s \to 0,$$
 where $\xi([\mu])=[\phi(\mu)-\mu]$ and $s$ is the number of components of $S$.
\end{lemma}
\begin{proof}
 Consider the triple $\big(S\times[0,1],\partial(S\times[0,1]),S\times\{0\}\big)$. Since $S\times[0,1]$ deformation retracts on $S\times\{0\}$, the homology of the corresponding pair is trivial. Moreover, there are equivalences $\big(S\times[0,1],\partial(S\times[0,1])\big)\sim(M,S)$ and $\big(\partial(S\times[0,1]),S\times\{0\}\big)\sim(S,\partial S)$. Finally $H_*(M,S)\cong H_{*-1}(S,\partial S)$. Hence the long exact sequence of the pair $(M,S)$ gives
 $$0 \to H_2(M) \to H_1(S,\partial S) \xrightarrow{\xi} H_1(S) \to H_1(M) \to 0.$$
 Finally, given an arc $a$ properly embedded in $(S,\partial S)$, $a\times[0,1]$ is a relative 2--cycle for the pair \mbox{$\big(S\times[0,1],\partial(S\times[0,1])\big)\sim(M,S)$} whose boundary is $-a\cup\phi(a)$.
\end{proof}

To compute the homology of $\partial X$, we need to understand the homology classes $\phi(\mu)-\mu$ in $H_1(\Sigma_1)$. We keep the notations defined before and in Proposition~\ref{prop:monodromy}.

\begin{proposition}
 Define families $\varepsilon_i$ in $H_1(\Sigma)$ as follows: $\varepsilon_1=0$ and $\varepsilon_{i+1}=\varepsilon_i+R_ia_i$. Fix a basis $b_L$ of $\L_1\cong L_1$ and complete it into a basis $(b_L,b)$ of $J_1$. Write $e$ in the basis $(b_L,e)$ of $\J_1$ and $\varepsilon_{n+1}$ in the basis $(b_L,b)$ of $J_1$. 
 The homology of $\partial X$ is the homology of the complex
 $$0 \to \Z^s \xrightarrow{\scriptstyle{0}} \frac{\J_1}{\L_1} \xrightarrow{\xi} \frac{J_1}{L_1} \xrightarrow{\scriptstyle{0}} \Z^s \to 0,$$
 where $s$ is the number of components of $\Sigma$ and $\xi$ is given in the bases $e$ and $b$ by the matrix 
 $S=e^t\varepsilon_{n+1}$.
\end{proposition}
\begin{proof}
 The $\varepsilon_i$ represent the homology classes in $H_1(\Sigma)$ of the $e_i-e$. Throughout the algorithm of Proposition~\ref{prop:AlgoCGPC}, they are added curves as in Proposition~\ref{prop:monodromy}, but we now view the result in $H_1(\Sigma)$ at each step. 
\end{proof}

\section{Sample calculations} \label{sec:Exs}

\subsection*{Example 1}

\begin{figure}[htb]
\begin{center}
\begin{tikzpicture}[scale=0.7]
 \foreach \s in {-1,1} {\draw (-1,4) .. controls +(0,0.5*\s) and +(0,0.5*\s) .. (1,4);}
 \foreach \s/\d in {-1/,1/dashed} {\draw[\d] (-1,-4) .. controls +(0,0.5*\s) and +(0,0.5*\s) .. (1,-4);}
 \foreach \s in {-1,1} {\draw (\s,-4) .. controls +(0,1) and +(0,-2) .. (3*\s,0) .. controls +(0,2) and +(0,-1) .. (\s,4);}
 \newcommand{\trou}{(-1,0) ..controls +(0.5,-0.4) and +(-0.5,-0.4) .. (1,0) (-0.8,-0.1) ..controls +(0.6,0.3) and +(-0.6,0.3) .. (0.8,-0.1)}
 \draw[xshift=-1cm,rotate=-90] \trou;
 \draw[xshift=1cm,rotate=-90] \trou;
 \foreach \s/\n/\c in {1/red,-1/blue} {
 \draw[\c] (\s,0) ellipse (0.8 and 1.4);}
 \draw[red,->] (1.56,1) -- (1.5,1.1) node[right] {$\scriptstyle{\alpha_2}$};
 \draw[blue,->] (-0.3,1.4) node {$\scriptstyle{\beta_2}$} (-0.44,1) -- (-0.5,1.1);
 \foreach \s/\d in {-1/,1/dashed} {
 \draw[red,\d] (-1.3,0) .. controls +(-0.5,0.4*\s) and +(0.5,0.4*\s) .. (-3,0);}
 \draw[red,->] (-2.5,-0.25) node[below] {$\scriptstyle{\alpha_1}$} -- (-2.6,-0.22);
 \foreach \s/\d in {-1/,1/dashed} {
 \draw[blue,\d] (1.15,0) .. controls +(0.5,0.4*\s) and +(-0.5,0.4*\s) .. (3,0);}
 \draw[blue,->] (2.5,-0.25) node[below] {$\scriptstyle{\beta_1}$} -- (2.6,-0.22);
 \foreach \s/\d in {-1/,1/dashed} {
 \draw[green,\d] (-0.85,0) .. controls +(0.5,0.3*\s) and +(-0.5,0.3*\s) .. (0.7,0);}
 \draw[green,->] (-0.5,-0.15) node[below] {$\scriptstyle{\gamma_1}$} -- (-0.4,-0.2);
 \draw[green] (2.3,2) .. controls +(-0.8,0.4) and +(1,0) .. (0,2.5) .. controls +(-1.5,0) and +(0,1.5) .. (-2.2,0) .. controls +(0,-1.5) and +(-1.5,0) .. (0,-2.5) .. controls +(1.5,0) and +(0,-1.5) .. (2.2,-0.5) .. controls +(0,0.5) and +(0.4,-0.1) .. (1.1,0.2);
 \draw[dashed,green] (1.1,0.2) .. controls +(0.6,0.3) and +(-0.4,0.1) .. (2.97,0.4);
 \draw[green] (2.97,0.4) .. controls +(-0.5,0.2) and +(0.5,0.1) .. (0.96,0.7);
 \draw[dashed,green] (0.96,0.7) .. controls +(0.4,0.6) and +(-0.6,-0.3) .. (2.3,2);
 \draw[green,->] (-1,2.3) -- (-1.1,2.25) node[above] {$\scriptstyle{\gamma_2}$};
 \draw[purple,->] (0,-3) node[right] {$\scriptstyle{e}$} -- (0,3.6) (0,-4.4) -- (0,-3);
\end{tikzpicture}
\end{center}
\caption{A trisection diagram of a disk bundle over $S^2$ with Euler number $-2$}
\label{figex1}
\end{figure}
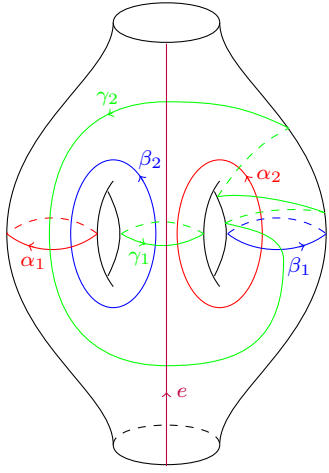

The trisection diagram $(\Sigma;\alpha,\beta,\gamma)$ in Figure~\ref{figex1} is a diagram for a disk bundle $X$ over $S^2$ with Euler number $-2$ obtained by Castro--Gay--Pinz\'on-Caicedo in \cite[Section~5.1]{CGPC1}.
In this example, all homology groups have coefficients in $\Z$. We first compute the (relative) homology and intersection form of $X$ from this diagram. 

In $H_1(\Sigma)=\langle\alpha_1,\beta_1,\alpha_2,\beta_2,\gamma_1\rangle$, we have $L_\alpha=\langle\alpha_1,\alpha_2\rangle$, $L_\beta=\langle\beta_1,\beta_2\rangle$, $L_\gamma=\langle\gamma_1,\alpha_2-2\beta_1+\beta_2\rangle$. All pairwise intersections of these subgroups are trivial. The homology of $X$ is the homology of the complex
$$0\to L_\alpha\oplus L_\beta\oplus L_\gamma \to H_1(\Sigma)\xrightarrow{\scriptstyle{0}} \Z,$$
giving $H_1(X)=0$, $H_2(X)\cong\Z$ and $H_3(X)=0$.

In $H_1(\Sigma,\partial\Sigma)=\langle\alpha_1,\beta_1,\alpha_2,\beta_2,e\rangle$, we have $\J_\alpha=\langle\alpha_1,\alpha_2,e\rangle$, $\J_\beta=\langle\beta_1,\beta_2,e\rangle$, $\J_\gamma=\langle\alpha_1-\beta_1,\alpha_2-2\beta_1+\beta_2,e-\beta_2\rangle$. Pairwise intersections are $\J_\alpha\cap\J_\beta=\langle e\rangle$, $\J_\beta\cap\J_\gamma=\langle e-\beta_2\rangle$, $\J_\gamma\cap\J_\alpha=\langle \alpha_1-\alpha_2-e\rangle$. The relative homology of $X$ is the homology of the complex
$$0\to\Z\xrightarrow{\scriptstyle{0}}\oplus_{\nu\neq\nu'}\J_\nu\cap\J_{\nu'} \to \oplus_{\nu}\J_\nu \to H_1(\Sigma,\partial\Sigma)\xrightarrow{\scriptstyle{0}} \Z\to0,$$
where $\nu,\nu'\in\{\alpha,\beta,\gamma\}$, giving $H_1(X,\partial X)=0$, $H_2(X,\partial X)\cong\Z$ and $H_3(X,\partial X)=0$.

A generator of $H_2(X)$ is given by $(\alpha_2,\beta_2-\beta_1,-\gamma_2)\in L_\alpha\oplus L_\beta\oplus L_\gamma$, while a generator of $H_2(X,\partial X)$ is $(\alpha_1,-\beta_1,\beta_1-\alpha_1)\in\J_\alpha\oplus\J_\beta\oplus\J_\gamma$. On these generators, the value of the intersection form is $1$.

We now consider the monodromy of the open book on $\partial X$. We set $a_1=(\alpha_1,\alpha_2)$, $a_2=(\beta_1,\beta_2)$ and $a_3=(\gamma_1,\gamma_2)$. Starting with $R_0=0$ and $e_1=e$, we compute $R_1=\begin{pmatrix} 0&0 \end{pmatrix}$, so that $e_2=e$, then $R_2=\begin{pmatrix} 0&-1 \end{pmatrix}$ and $e_3=e-\beta_2$, finally $R_3=\begin{pmatrix} -2&1 \end{pmatrix}$ and $e_4=e-2\alpha_1+\alpha_2$. This gives $R=\begin{pmatrix} 1 \end{pmatrix}$ and shows that the action of the monodromy on $H_1(\Sigma_1,\partial\Sigma_1)$ is trivial. Now starting with $\varepsilon_1=0$, we get $\varepsilon_2=0$, $\varepsilon_3=-\beta_2$ and $\varepsilon_4=\alpha_2-2\alpha_1-2\zeta$. Hence $S=\begin{pmatrix} -2 \end{pmatrix}$. Finally, the homology of $\partial X$ is the homology of the complex
$$0 \to \Z \xrightarrow{\scriptstyle{0}} \langle e\rangle \xrightarrow{-2}\langle \zeta\rangle \xrightarrow{\scriptstyle{0}} \Z \to 0,$$
giving $H_1(\partial X)=\Z/2\Z$ and $H_2(\partial X)=0$.

\subsection*{Example 2}

\begin{figure}[htb]
\begin{center}
\begin{tikzpicture}[scale=0.7]
 \foreach \s in {-1,1} {\draw (7,2*\s) .. controls +(-6,0) and +(0,2*\s) .. (-3,0);}
 \draw (7,0) ellipse (0.5 and 2);
 \newcommand{\trou}{(-1,0) ..controls +(0.5,-0.4) and +(-0.5,-0.4) .. (1,0) (-0.8,-0.1) ..controls +(0.6,0.3) and +(-0.6,0.3) .. (0.8,-0.1)}
 \draw \trou;
 \draw[xshift=4cm] \trou;
 \foreach \s/\d in {-1/,1/dashed} {\draw[red,\d] (0,-0.3) .. controls +(0.3*\s,-0.5) and +(0.3*\s,0.5) .. (0,-1.8);}
 \draw[red,->] (0,-1.8) node[below] {$\scriptstyle{\alpha}$} (-0.2,-1.3) -- (-0.17,-1.4);
 \foreach \s/\d in {-1/,1/dashed} {\draw[\d] (4,-0.3) .. controls +(0.3*\s,-0.5) and +(0.3*\s,0.5) .. (4,-2);}
 \draw[->] (4,-2) node[below] {$\scriptstyle{x}$} (3.8,-1.5) -- (3.83,-1.6);
 \draw[blue] (0,0) ellipse (1.3 and 0.7);
 \draw[blue,->]  (0.1,0.7) -- (0,0.7);
 \draw[blue] (0.1,1) node {$\scriptstyle{\beta}$};
 \draw[->]  (4.1,0.7) -- (4,0.7) node[below] {$\scriptstyle{y}$};
 \draw (4,0) ellipse (1.3 and 0.7);
 \draw[green] (2,1.9) .. controls +(-1.5,-1) and +(0,1.5) .. (-2,0) .. controls +(0,-1) and +(-1,0) .. (0,-1).. controls +(0.5,0) and +(0,-0.5) .. (0.7,-0.2);
 \draw[green,dashed] (0.7,-0.2) .. controls +(2,-1) and +(0,-1.5) .. (6,0) .. controls +(0,1.5) and +(2,-1) .. (2,1.9);
 \draw[green,->] (-2,0.1) -- (-2,0) node[left] {$\scriptstyle{\gamma}$};
 \foreach \s in {-1,1} {\draw[purple] (6.6,1.2*\s) .. controls +(-2,0) and +(0,1.5*\s) .. (2,0);}
 \draw[purple,->] (6.4,1.2) -- (6.3,1.2) node[above] {$\scriptstyle{e'}$};
 \draw[purple] (6.5,0.2) .. controls +(-1,0) and +(0.2,0.5) .. (4.8,-0.1) (6.5,-0.4) -- (7.5,-0.4);
 \draw[purple,dashed] (4.8,-0.1) .. controls +(0.2,-0.5) and +(-1,0) .. (6.5,-0.4);
 \draw[purple,->] (6.4,0.2) -- (6.3,0.2) node[above] {$\scriptstyle{e}$};
\end{tikzpicture}
\end{center}
\caption{A trisection diagram}
\label{figex2}
\end{figure}
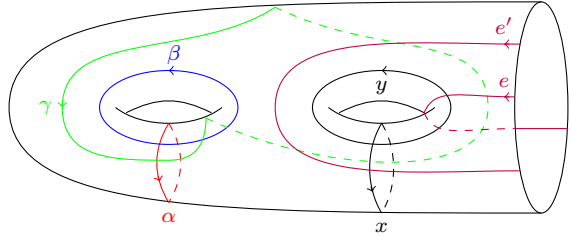

Let $X$ be the $4$--manifold defined by the trisection diagram $(\Sigma;\alpha,\beta,\gamma)$ in Figure~\ref{figex2}.

In $H_1(\Sigma)=\langle\alpha,\beta,x,y\rangle$, we have $L_\alpha=\langle\alpha\rangle$, $L_\beta=\langle\beta\rangle$, $L_\gamma=\langle -\alpha+\beta+y\rangle$. Pairwise intersections are trivial and we get $H_1(X;\Z)=\langle x\rangle\cong\Z$, $H_2(X;\Z)=0$ and $H_3(X;\Z)=0$.

In $H_1(\Sigma,\partial\Sigma)=\langle\alpha,\beta,e,e'\rangle$, we have $\J_\alpha=\langle\alpha,e,e'\rangle$, $\J_\beta=\langle\beta,e,e'\rangle$, $\J_\gamma=\langle\beta-\alpha,\alpha+e,e'\rangle$. It gives $H_1(X,\partial X)=0$, $H_2(X,\partial X)=0$ and $H_3(X,\partial X)=\J_\alpha\cap\J_\beta\cap\J_\gamma\cong\Z$.

We define $\varphi: \Z[\pi_1(X,*)]\to\Zt$ by $\varphi(x)=t$. Let us compute the associated twisted homology and torsion. Fix a lift $\tilde *$ of the basepoint $*$. For $\zeta\in\pi_1(\Sigma,*)$, we denote $\tilde\zeta$ the lift of~$\zeta$ starting at $\tilde *$. Since $\gamma=\alpha^{-1}xyx^{-1}\alpha\beta\alpha^{-1}$ in $\pi_1(\Sigma,*)$, we have $\tilde\gamma=-\tilde\alpha+\tilde\beta+t\tilde y$ in $H_1^\varphi(\Sigma,*)$. Hence, in $H_1^\varphi(\Sigma,*)=\langle\tilde\alpha,\tilde\beta,\tilde x,\tilde y\rangle$, we have $L_\alpha^\varphi=\langle\tilde\alpha\rangle$, $L_\beta^\varphi=\langle\tilde\beta\rangle$, $L_\gamma^\varphi=\langle -\tilde\alpha+\tilde\beta+t\tilde y\rangle$. From this complex:
$$0\to L_\alpha^\varphi\oplus L_\beta^\varphi\oplus L_\gamma ^\varphi\to H_1^\varphi(\Sigma,*)\to H_0^\varphi(*)\to0$$
we get $H_0^\varphi(X;\Zt)\cong\Zt/(t-1)\cong\Z$ and $H_i^\varphi(X;\Zt)=0$ for $i>0$. It implies that the homology of $X$ with coefficients in $\Q(t)$ is trivial, so that the torsion won't depend on the choice of a homology basis. 
Set $c_2=(\tilde\alpha,\tilde\beta,\tilde\gamma)$, $c_1=(\tilde\alpha,\tilde\beta,\tilde x,\tilde y)$ and $c_0=(\tilde *)$ as complex bases for the above complex and $b_1=(\tilde\alpha,\tilde\beta,\tilde\gamma)$ and $b_0=((t-1)\tilde *)$ as bases of the images of the boundary map. Then the torsion is given by
$$\tau^\varphi(X)=\left[\frac{b_1}{c_2}\right]^{-1}\left[\frac{b_1b_0}{c_1}\right]\left[\frac{b_0}{c_0}\right]^{-1}=-t(t-1)^{-1}\in\Q(t)/\Zt.$$

Finally, we consider the monodromy of the open book on~$\partial X$. We set $a_1=\alpha$, $a_2=\beta$ and $a_3=\gamma=e'-\alpha+\beta$. Starting with $R_0=0$ and $e_1=\begin{pmatrix} e \\ e' \end{pmatrix}$, we get $R_1=\begin{pmatrix} 0 \\ 0 \end{pmatrix}$ and $e_2=e_1$, then $R_2=\begin{pmatrix} 1 \\ 0 \end{pmatrix}$ and $e_3=\begin{pmatrix} e+\beta \\ e' \end{pmatrix}$, finally $R_3=\begin{pmatrix} -1 \\ 0 \end{pmatrix}$ and $e_4=\begin{pmatrix} e-e'+\alpha \\ e' \end{pmatrix}$. Utilizing the basis $(e,e',\alpha)$ of $\J_1$, we obtain 
$$e^te_4=\begin{pmatrix} 1 & 0 & 0 \\ 0 & 1 & 0 \end{pmatrix}\begin{pmatrix} 1 & 0 \\ -1 & 1 \\ 1 & 0 \end{pmatrix}=\begin{pmatrix} 1 & 0 \\ -1 & 1 \end{pmatrix}$$
as the matrix giving the action of the monodromy in the basis $(e,e')$ of $H_1(\Sigma_1,\partial\Sigma)$.

To get the homology of $\partial X$, we start with $\varepsilon_1=(0,0)$ and the computation gives $\varepsilon_2=(0,0)$, $\varepsilon_3=(\beta,0)$, $\varepsilon_4=(\alpha-y,0)$. It follows that the homology of $\partial X$ is the homology of the complex 
$$0\to\Z\xrightarrow{\scriptstyle 0}\langle e,e'\rangle\xrightarrow{\xi}\langle x,y\rangle\xrightarrow{\scriptstyle 0}\Z\to0,$$
where $\xi(e)=-y$ and $\xi(e')=0$. Thus $H_1(\partial X)\cong H_2(\partial X)\cong\Z$.

\def\cprime{$'$}
\providecommand{\bysame}{\leavevmode\hbox to3em{\hrulefill}\thinspace}
\providecommand{\MR}{\relax\ifhmode\unskip\space\fi MR }
\providecommand{\MRhref}[2]{%
  \href{http://www.ams.org/mathscinet-getitem?mr=#1}{#2}
}
\providecommand{\href}[2]{#2}

\end{document}